\documentclass{amsart}

\usepackage{amsmath}

\usepackage{amssymb}
\usepackage{amscd}

\numberwithin{equation}{section}

\newif\ifdraft\drafttrue
\draftfalse

\long\def\comdmitry#1{\ifdraft{\marginpar{\sn
#1 \ DK}}\else\ignorespaces\fi}

\font\sn = cmssi8 scaled \magstep0

\newcommand\ba{badly approximable}

\newcommand\da{Diophantine approximation}

\newcommand\hs{homogeneous space}

\newcommand\ssm{\smallsetminus}

\newcommand\name[1]{\label{#1}{\ifdraft{\sn [#1]}\else\ignorespaces\fi}}

\newcommand\eq[2]{{\ifdraft{\ \tt [#1]}\else\ignorespaces\fi}\begin{equation}\label{eq:#1}{#2}\end{equation}}

\newcommand {\equ}[1]     {\eqref{eq:#1}}

\newcommand{\absd}{absolutely decaying}

\newcommand\z{\mathbf{z}}

\newcommand{\Q}{{\mathbb {Q}}}

\newcommand{\Hyp}{{\mathbb {H}}}

\newcommand{\R}{{\mathbb{R}}}

\newcommand{\T}{{\mathbb{T}}}

\newcommand{\Z}{{\mathbb{Z}}}

\newcommand{\N}{{\mathbb{N}}}

\newcommand{\vv}{{\bf{v}}}

\newcommand{\GL}{\operatorname{GL}}
\newcommand{\Lip}{\operatorname{Lip}}

\newcommand{\Isom}{\operatorname{Isom}}

\newcommand{\diam}{\operatorname{diam}}

\newcommand{\dist}{\operatorname{dist}}

\newcommand{\supp}{\operatorname{supp}}

\newcommand {\ignore}[1]  {}

\newcommand{\df}{{\, \stackrel{\mathrm{def}}{=}\, }}

\newcommand{\x}{{\mathbf{x}}}

\newcommand{\vu}{{\bf u}}

\newcommand{\vm}{{\bf m}}

\newcommand{\p}{{\bf p}}

\newcommand{\vp}{{\bf p}}

\newcommand{\BA}{{\bold{BA}}}

\newcommand\hd{Hausdorff dimension}

\newcommand{\rd}{\mathbb{R}^d}
\newcommand{\rn}{\rd}

\newcommand{\zn}{\mathbb{Z}^n}

\newcommand{\y}{\mathbf{y}}

\newcommand{\cl}{\mathcal L}

\newtheorem{thm}{Theorem}[section]
\newtheorem{lem}[thm]{Lemma}
\newtheorem{prop}[thm]{Proposition}
\newtheorem{cor}[thm]{Corollary}
\newtheorem*{corcor}{Corollary}
\newtheorem{example}[thm]{Example}

\newtheorem{defn}[thm]{Definition}

\begin{document}

\title[${\bf{BA}}$ is strongly  $C^1$ incompressible]{The set of badly approximable vectors
%DK
 \\ is strongly $C^1$ incompressible}
\author[Broderick, Fishman, Kleinbock, Reich, Weiss]{Ryan Broderick, Lior Fishman, Dmitry Kleinbock, Asaf Reich and Barak Weiss}

\address{Department of Mathematics, Brandeis University, Waltham MA 02454-9110, USA} 
\email{ryanb@brandeis.edu,  lfishman@brandeis.edu,  kleinboc@brandeis.edu, azmreich@brandeis.edu}

\address{Department of Mathematics, Ben Gurion University, Be'er Sheva, Israel 84105} 
\email{barakw@math.bgu.ac.il}

\date{June 4, 2011}

%\comdmitry{Didn't like the old title too much, so came up with this, changed non-uniformly to strongly.}

\maketitle{}

\begin{abstract}
We prove that the countable intersection of $C^1$-diffeomorphic
images of certain Diophantine sets has full Hausdorff
dimension.  For example, we 
%DK
show this for the set of badly
approximable vectors in $\rd$, improving earlier results of Schmidt
and Dani. To prove this, inspired by ideas of
McMullen, we define a new variant of Schmidt's 
%DK
$(\alpha,
\beta)$-game and
show that our sets are {\em hyperplane absolute winning\/} 
%DK
(HAW),
which in particular implies winning in 
the original game. The HAW property passes automatically to games
played on 
%DK
certain
fractals, thus our sets intersect a large class of fractals
in a set of positive dimension. This extends 
%DK
earlier results of Fishman to
a more general 
%DK
%class of fractals, 
set-up, with
simpler proofs. 
\end{abstract}
%\comdmitry{Haven't changed the abstract yet.}
\section{Introduction}

We begin with a definition 
%DK
introduced by S.G.\ Dani \cite{Dani conference}:
a subset $S$ of $\rd$ is called {\sl incompressible\/} if for any
nonempty open $U\subset \rd$ and any  uniformly bi-Lipschitz sequence
$\{f_i\}$ of maps of $U$ onto (possibly different) open subsets of
$\rd$, the set 
\eq{inters}{\bigcap_{i = 1}^\infty f_i^{-1}(S)}
has   \hd\ $d$. Here a sequence of maps $\{f_i: U\to \rd\}$ is called
{\sl  uniformly bi-Lipschitz\/} if  
\eq{bilip}{\sup_i \Lip(f_i) < \infty\,,}
where $\Lip(f)$ is the bi-Lipschitz constant of $f$, defined to be the infimum of  $C \ge  1$ such that 
%\eq{bilipconst}{
$C^{-1}\|\x-\y\| \le\|f(\x) - f(\y)\|\le C \|\x-\y\|$ for all $\x,\y\in U$.
%}
%for all $\x,\y\in U$ and $i\in \N$. 
%This definition is clearly independent of the choice of the 
Here and below  $\|\cdot\|$ stands for the Euclidean norm
on $\rd$.
%, which we will fix to be Euclidean norm from now on. 
%A (norm-dependent) value of $C$ for which \equ{bilip} holds is
%referred to as a uniform bi-Lipschitz constant for the sequence. 
One of the main interesting examples of incompressible sets is the set ${\bf{BA}}_d$
%\comdmitry{Shouldn't we call it ${{\bf{BA}}}_d$ to avoid confusion?} 
of {\sl badly approximable vectors\/}, namely those $\x \in \rd$ satisfying
\eq{def ba}{ \left\|\x - \frac{\p}{q} \right\| \geq \frac{c}{q^{1+1/d}} }
for some $c=c(\x)>0$ and for any $\p \in \mathbb{Z}^{d}$, $q\in
\mathbb{N}$. %, which is a countable union of nowhere dense  sets of
             %measure zero.  
Its incompressibility is a consequence of a stronger property: namely,
of the fact that ${\bf{BA}}_d$ is a winning set of  
%DK
a game introduced in the 1960s by W.M.\ Schmidt \cite{S1, S3} (see \S\ref{games} for definitions and a discussion).

Our goals in this paper are to remove
the hypothesis of uniformity in \equ{bilip}
and to consider intersections of bi-Lipschitz images of $S$ with certain 
closed  subsets $K\subset\rd$.  Let us say that $S\subset\rd$ is  {\sl strongly\/} 
(resp., {\sl strongly $C^1$, strongly affinely\/}) {\sl incompressible on $K$\/} if 
%\equ{interswithk}
%holds 
 one has % the \hd\ of the set
\eq{interswithk}{\dim\left(\bigcap_{i = 1}^\infty f_i^{-1}(S)\cap
  K\right) = \dim(U\cap K)} 
for any  open $U\subset \rd$ with $U\cap K\ne\varnothing$ and any
sequence $\{f_i\}$ of  bi-Lipschitz  maps (resp., $C^1$
diffeomorphisms, affine nonsingular maps) of $U$ onto (possibly
different) open subsets of $\rd$ (here and hereafter $\dim$ stands for
the \hd).  
 %When $K = \rd$ the reference to it will be omitted.
 
 %, as before.
%, incompressibility on $\br^d$ will be simply referred to as incompressibility. 
It is easy to see that strong incompressibility (in $\rd$) is shared
by subsets of full Lebesgue measure in $\rd$: indeed, if $\rd\ssm S$
has  
Lebesgue measure zero, then the set \equ{inters} has full measure in $U$, and hence full \hd. 
%DK
See also \S\ref{vwa} for an example of a strongly incompressible residual subset of $\rd$
of Lebesgue measure zero.
%It follows from Baire's theorem that if $S$ is residual, then so
%is the intersection \equ{inters}; showing incompressibility of residual
%sets, i.e. showing that the intersection has full Hausdorff dimension,
%is related to work of Falconer and Durand which we discuss at the
%end of the paper. 
%DK
On the other hand, using Schmidt
games one can exhibit 
%DK
%Lebesgue 
%measure zero 
sets with strong intersection properties
which are Lebesgue-null  and meager, i.e.\ are small
from the point of view of both measure and category. 
Indeed, it
follows from Schmidt's work that 
%DK
 any winning subset  of $\R^d$ is
incompressible, and
strongly $C^1$ incompressible if $d = 1$ (the reason being that 
%DK bi-Lipschitz 
diffeomorphisms in
dimension $1$ are, in Schmidt's 
terminology, {\it local isometries\/}).

In a more recent development, it was shown  by the second-named author
\cite{F} that the set ${\bf{BA}}_d$  is   strongly affinely
incompressible. Moreover, the main result of \cite{F} establishes
strong affine incompressibility of  ${\bf{BA}}_d$ on various fractal
subsets of $\rd$, characterized by their ability to support measures 
satisfying certain decay conditions. 
In particular it is proved there that ${\bf{BA}}_d$ is strongly
affinely incompressible on $K = \supp\,\mu$  whenever $\mu$ is a
measure on $\rd$ which is {\it absolutely decaying\/}  
and {\it Ahlfors regular\/}. Examples of such sets  include limit sets of irreducible families of
self-similar contractions of $\rd$ satisfying the open set condition; see \S\ref{meas} for more examples and precise definitions.
%(we postpone  the definitions of these properties until ).

 In this paper we generalize the aforementioned result as follows:

%by means of introducing a class of subsets of $\rd$ which includes
 %all examples considered in \cite{F}  and other related papers. Namely, let us say that a closed subset $K$ of $\rd$ is  {\sl  hyperplane diffuse\/}  if there
%exists $\rho_K, \beta > 0$ such that
%for any $0 < \rho \le \rho_K$, $\x\in K$, and any affine hyperplane $\cl$, 
%there exists $\x' \in K$ such that
%$$B(\x',\beta\rho) \subset B(\x,\rho)\ssm \cl^{(\beta\rho)}.$$
%Here $B(\x,\rho)$ is the open ball centered at $x$ with radius $\rho$, and  $\cl^{(\varepsilon)$ is 
%the $\varepsilon$-neighborhood of $\cl$.
  
%We exhibit   Here is one of our main results:
  
 \begin{thm}
\label{incompr} Let $\mu$ be a measure on $\rd$ which is  absolutely decaying 
and   Ahlfors regular. Then  ${\bf{BA}}_d$  is   strongly $C^1$ incompressible on $K = \supp\,\mu$. \end{thm}

%Whether or not the same is true without the additional restriction of smoothness, even for 
%$K = \rd$,  is an open question.

Our results are in fact much more general. We use ideas from a recent paper \cite{Mc} of McMullen,
%DK
where, among other things, it was proved that ${\bf{BA}}_1$ is strongly incompressible, and introduce a modification
of games considered therein.
% \cite{Mc, S1, S3}. 
More precisely, we show that the set ${\bf{BA}}_d$ has a
property which we  refer to as {\it hyperplane absolute winning\/}, 
%\comdmitry{Not sure if this is such a good abbreviation, feel free to suggest an alternative.}
abbreviated by 
%DK
HAW; see \S\ref{games} for a definition. This property is stable under countable intersection and implies winning
in Schmidt's sense. One of the main results of the present  paper (Theorem \ref{stability}) is that the HAW property is preserved 
by $C^1$ diffeomorphisms; this immediately implies  the case $K = \rd$
of Theorem  \ref{incompr}.  Furthermore, a rather elementary
observation, see Proposition \ref{cap K},  shows that the HAW property  
%is inherited by 
 %sets $K$ as in Theorem \ref{incompr},
%in fact under  much less restrictive assumptions on $K$ which we discuss in \S\ref{games}. That is, 
%the fact that $S$ is HAW 
implies 
a similar property defined with respect to a game played on $K$,
%\comdmitry{As you see, here (i.e.\ in the introduction) I decided to focus on full dimension results; we can discuss weaker conclusions in the main body of the paper.} 
where $K$ is a closed subset of $\rd$ satisfying much less restrictive assumptions that in Theorem \ref{incompr}. More precisely, in \S\ref{fr} we introduce a class of  {\it hyperplane diffuse\/} sets which includes
those satisfying assumptions of Theorem \ref{incompr}, 
and prove the following 
%, which is a special case of more general results:

 \begin{thm}
\label{incomprdiffuse} Let $K\subset \rd$ be hyperplane  diffuse,  $U$
an open subset of $\rd$ with 
$U\cap K\ne \varnothing$, $S\subset \rd$ hyperplane absolute winning,
and let $\{f_i\}$ be a sequence of $C^1$ diffeomorphisms of $U$ onto
(possibly different) open subsets of $\rd$. Then the set 
\eq{intersdiffuse}{\bigcap_{i = 1}^\infty f_i^{-1}(S)\cap K}
has positive \hd.\end{thm}

This theorem in particular applies to $S = {\bf{BA}}_d$; other
examples of HAW sets are discussed later in the paper.

The structure of the paper is as follows: in \S\ref{games} we first
describe Schmidt's game in its original form, 
then its modification introduced by McMullen, and then introduce a new variant.
  In \S\ref{stab} we prove Theorem \ref{stability}, which establishes
  the invariance of the new winning property under nonsingular smooth
  maps, and also show that $\BA_d$ is hyperplane absolute winning,
  thereby proving 
the statement made in the title of the paper. A remarkable feature of
hyperplane absolute winning sets, is that they have a large
intersection with any set from a large class of fractals. To this end,
in \S\ref{fr} we explain 
how the games defined in \S\ref{games} are played on proper subsets
of $\rd$, introduce all the necessary definitions 
to distinguish a class of  
%DK
diffuse  fractals on which those games
can
be played successfully, and prove Proposition \ref{prop: intersection} which implies
that diffuse sets intersect hyperplane absolute winning sets
nontrivially. From this we 
derive Theorem \ref{incomprdiffuse}. 
A discussion of measures whose supports can be shown to have the
diffuseness property takes place in \S\ref{meas} and leads to the
proof of Theorem \ref{incompr}. 
The last section is devoted to some concluding remarks and several open questions. 
%Also it will be nice to have a section with open questions, we have quite a few of those. 
%{\sb Need to finish this after going through the whole text.}

\ignore{
sets which were previously shown to be winning can 
introduce a class of subsets of $\rd$ which we call {\sl
  hyperplane-diffuse\/}, and show that those always have  

Recently C.\ McMullen \cite{MC} introduced and investigated the
stability of various winning properties based on Schmidt's game
(\cite{S1}), called strong winning and absolute winning, under
quasiconformal maps (see section \ref{games} for the definitions.) 
However, his notion of ``absolute winning" is too restrictive to apply to some well-known winning sets. For example, 
In this paper we define a generalized variant of his absolute winning sets, ``$\kappa$-absolute winning", that does apply to ${\bf{BA}}_d$, and which possesses more stability features than strong winning sets.
Namely we show
\begin{cor} If $f_n$ is any sequence of $C^1$-diffeomorphisms of
  $\rd$, and ${\bf{BA}}_d$ is the set of badly approximable vectors,
  then $\cap f_n({{\bf{BA}}_d})$ is winning, thus of full
  dimension. \end{cor} 

\noindent In fact, an advantage of our technique is that a much more
general result immediately follows: 
\begin{corcor}$\bf{1.1'.}$ Let $f_n$ be any sequence of $C^1$-diffeomorphisms of $\rd$, and $K$ any sufficiently regular fractal. Then $\cap f_n({{\bf{BA}}_d})$ is winning on $K$ and thus of positive and often full dimension.
\end{corcor}}

\section{Games}
\label{games}
Consider the game, originally introduced by 
%DK
Schmidt \cite{S1}
where two players, whom we will call Bob and Alice, successively
choose nested closed balls in $\rd$ 
%\comdmitry{I reindexed the sequence, making it start with $1$, is this OK?}
%\comliorryan{Changed the indexing in \equ{radii}, because $A_{i-1}$ is not defined for $i=1$.
%Nothing is said about $\rho(B_1)$, which is technically fine since it can be arbitrary, but is it clear?}
\eq{balls}{ B_1 \supset A_1 \supset B_2 \supset\dots} 
satisfying
\eq{radii}{
\rho(A_i)=\alpha\rho(B_i), \ \ \rho(B_{i+1})=\beta\rho(A_{i}) 
}
for all $i$, where $\rho(B)$ denotes the radius of $B$, and
$0<\alpha,\beta<1$ are fixed. 
A set $S \subset \rd$ is said to be {\sl $(\alpha,\beta)$-winning\/}
if Alice has a strategy guaranteeing that  
the unique point of intersection 
$\bigcap_{i=1}^\infty B_i= \bigcap_{i=1}^\infty A_i$ of all the balls
belongs to $S$,
%\eq{Alicewins}{\bigcap_{i=1}^\infty B_i\cap S\ne\varnothing} 
regardless of the way Bob chooses to play. It is said to be {\sl
  $\alpha$-winning\/} if it is $(\alpha,\beta)$-winning for all
$\beta>0$, and {\sl winning\/} if it is $\alpha$-winning for some 
$\alpha$.

In \cite{Mc} McMullen introduced the following modification:
%in the case that $K=\rd$: 
\ignore{ if Alice still has a winning strategy with \equ{radii} replaced by
$$\rho(A_i) \geq \alpha\rho(B_i), \ \ \rho(B_{i+1}) \geq \beta\rho(A_{i}), $$
$S$ is called {\sl $(\alpha,\beta)$-strong winning\/}, with  {\sl
    $\alpha$-strong winning\/} and {\sl strong winning\/} 
defined accordingly. }%\comdmitry{For now I removed
%(If the radii do not tend to $0$ the notion is trivial since it reduces to $S$ being dense.)
%\comdmitry{So I removed strong winning, perhaps it can appear at the end of the paper, if we need it...}
Bob and Alice choose balls $B_i$, $A_i$ so that
\[ B_1 \supset B_1 \ssm A_1 \supset B_2\supset B_2 \ssm A_2 \supset
B_3 \supset \cdots,\]
and $$ \rho(A_i) \leq \beta\rho(B_i), \ \ \rho(B_{i+1}) \geq
\beta\rho(A_{i}), $$ where $\beta<1/3$ is fixed. 
One says that $S\subset \rd$ is {\sl $\beta$-absolute winning\/} if
Alice has a strategy which leads to 
\eq{Alicewins}{\bigcap_{i=1}^\infty B_i\cap S\ne\varnothing} 
regardless of how Bob chooses to play; $S$ is said to be  {\sl
  absolute winning\/} if it is $\beta$-absolute winning for all $0 <\beta <1/3$. 
Note a significant difference between this modification and the original
rule: now Alice has rather limited control over the situation, since 
she can block fewer of Bob's possible moves at the next step. Also, in
the new version the radii of balls do not have to tend to zero,
therefore $\cap_i B_i$ 
does not have to be a single point (however 
the outcome with radii not tending to $0$ is clearly winning for Alice as long as $S$ is dense).

\ignore{We  remark\comdmitry{I am not sure whether or not we need such
    a remark, and if yes how to write it.} 
 that in the original Schmidt's definition and \equ{Alicewins} amounts
 to saying that this point belongs to $S$.  
With McMullen's modification, as well as with those introduced later in this paper, this does not have to be the case; }

The following proposition summarizes some important properties of winning %, strong winning 
and absolute winning subsets of $\rd$:
\begin{prop}\label{properties}
\begin{itemize}
\item[(a)] Winning sets are dense and have  Hausdorff dimension $d$. 
\item [(b)] Absolute winning implies %$\alpha$-strong winning $\forall\,\alpha<1/2$, and $\alpha$-strong winning implies 
$\alpha$-winning for all $\alpha<1/2$. 
\item[(c)]  The countable intersection of $\alpha$-winning (resp., %$\alpha$-strong winning, 
absolute winning) sets is again $\alpha$-winning (resp.,  %$\alpha$-strong winning, 
absolute winning). 
\item  [(d)] The image of an  $\alpha$-winning  set under a bi-Lipschitz map $f$ is  $\alpha'$-winning, where $\alpha'$ depends only on $\alpha$ and $\Lip(f)$. 
\item  [(e)] %The image of an   $\alpha$-strong winning set under an $M$-quasisymmetric map $f$ is  $\alpha'$-strong winning, where $\alpha'$ depends only on $\alpha$, $d$ and $M$. 
%\item  [6.] 
The  image of an absolute winning set under a quasisymmetric map is absolute winning. 
\end{itemize}\end{prop}
For the proofs, see \cite{S1} and \cite{Mc}. We recall that $f$ is called {\sl $M$-quasisymmetric\/} if \eq{qs}{\begin{aligned}\forall\,r > 0\ \exists\,s > 0\text{  such that   for all balls }
B(\x, r) \subset \rd\\\text{ one has }
B\big(f(\x), s) \subset f\big(B(\x, r)\big) \subset B\big(f(\x), Ms\big)\,,
\end{aligned}}
and {\sl quasisymmetric\/} if it is $M$-quasisymmetric for some $M$.
%and $\|f\|_{QS}$ is defined to be the smallest $M$ for which \equ{qs} holds.
Note that being  quasisymmetric  is equivalent to being quasiconformal when $d \geq 2$, see
\cite[Theorem 7.7]{G};
%\comdmitry{We should use references from \cite{Mc}.} 
%\comliorryan{Done}
%DK
also, bi-Lipschitz clearly implies  quasisymmetric, but not vice versa. 
The above proposition 
can be used to establish the following incompressibility properties of winning sets:

\begin{prop}\label{winningincompr}
\begin{itemize}
\item[(a)] Winning sets are incompressible.
\item [(b)] Absolute winning sets are strongly incompressible.
\end{itemize}
%\comdmitry{The point of writing  proofs of both parts is that those properties are
%local and not completely straightforwardly follow from Proposition \ref{properties}. Perhaps for that we'll need
%to introduce the notion of winning on subsets of $\rd$ here already, and not in the next section as I planned.}
%&1. \text{ Winning sets are incompressible.} \notag \\
%&2. \text{ Absolute winning sets are strongly incompressible.} \notag 
%\end{align}
\end{prop}

Part (a) was proved by Dani in \cite{Dani conference}, and (b) follows
from one of the 
main results of \cite{Mc}. %For the sake of self-containment, we
                           %present the proof of both parts in
                           %\S\ref{proofs}.  
One can see now that  one of the advantages of proving a set to be
absolute winning is a possibility to intersect it with its countably
many bi-Lipschitz  images without worrying about  
a uniform bound on bi-Lipschitz  constants.
%{\sb We should ask a question about a quasisymmetric analogue of all
%that at the end of the paper.} 

Note that winning sets arise naturally in many  
settings in dynamics and Diophantine approximation  \cite{Dani-rk1, D,
  Dani conference, ET, Fae, FPS, KW2, KW3, Mo, S1, S2, T1, T2}. %{\sb
                                %(we should have more references
                                %here)} 
%and in many cases those sets can be shown to be  strong 
%winning. 
Several
examples of absolute winning sets 
were exhibited by McMullen in \cite{Mc}, most notably the set of badly
approximable numbers in $\R$. 
However absolute winning does not occur as frequently;  see \cite{Mc}
%McMullen found
for examples of sets in $\R$ 
which are winning but not absolute winning.  
\ignore{
In
particular, he showed that the set ${\bf BA}\subset \R$ of 
\ba\ numbers is absolute winning. More generally, he proved

\begin{prop}\label{bddgeod}
 The set $D(\Gamma)$ of 
endpoints of lifts of all bounded geodesic rays in $\Hyp^{d+1}
/\Gamma$, where $\Gamma  \subset \Isom(\Hyp^{d+1} )$ is a lattice with
a cusp at infinity, is absolutely winning. 
\end{prop}

This strengthened Dani's  result \cite{Dani-rk1} on the 
winning property of the set $D(\Gamma)$.}
Another example is
the set of \ba\ vectors in $\rd$ for $d > 1$. Indeed, the hyperplane
$\{0\} \times \mathbb{R}^{d-1}$ is 
disjoint from  ${\bf BA}_d$, thus Alice must avoid it in order to win
the game.  
However it is clear that Bob can play the absolute game in such a way
that his balls are always centered on this hyperplane. 

The above example 
suggests that in order to build a version of an absolute game adapted for treatment
of the set of badly approximable vectors, one needs to allow Alice to
`black out' neighborhoods of hyperplanes. 
This is precisely how we proceed to define our new version of the
game. Namely, fix $k \in\{0,1,\dots,d-1\}$ 
and $ 0 < \beta < 1/3$, and define the {\sl $k$-dimensional $\beta$-absolute game\/}
%\comdmitry{As you see, I did not like $\kappa$ and changed the
%terminology to avoid references to pairs $(\kappa,\beta)$ so that
%they are not confused with $(\alpha,\beta)$, and also replaced
%codimension with dimension.}  
in the following way. Bob initially chooses $\x_1\in \rd$ and $\rho_1>
0$, thus defining a closed ball $B_1 = B(\x_1,\rho_1)$. 
Then in each stage of the game, after Bob chooses $\x_i \in \rd$ and
$\rho_i > 0$, Alice chooses an affine subspace $\cl$ of dimension $k$
and removes its $\varepsilon\rho_i$-neighborhood $A_i\df
\cl^{(\varepsilon\rho_i)}$ from  
$B_i \df B(\x_i,\rho_i)$ for some $0 < \varepsilon \leq \beta$ (note
that $\varepsilon$ is allowed to depend on $i$). Then Bob chooses
$\x_{i+1}$ 
and $\rho_{i+1} \ge \beta\rho_i$ such that
$$B_{i+1}\df B(\x_{i+1},\rho_{i+1}) \subset B_i \ssm A_i\,.$$
The set $S$ is said to be {\sl $k$-dimensionally $\beta$-absolute  winning\/} if 
%for every positive 
%$\beta < 1/3$ 
Alice has a strategy guaranteeing that $\bigcap_i B_i$ intersects $S$. 
Note that our convention on \emph{intersecting} $S$ is in agreement
with McMullen's approach and simplifies our proofs, but is slightly
different from Schmidt's definition of general games in \cite{S1},
which requires 
the intersection of balls to be a subset of $S$. 
Also note that %, as in Schmidt's original game, 
\begin{equation}\label{comparison}
\begin{aligned}
\text{$k$-dimensional $\beta$-absolute winning implies
  $k$-dimensional  $\beta'$-absolute  winning}\\ \text{ whenever
  $\beta' \ge \beta$, as long as $ \beta' <
  1/3$}\,.\qquad\qquad\qquad\qquad\end{aligned} 
\end{equation}
We will say that $S$ is  {\sl $k$-dimensionally  absolute  winning\/} if it is $k$-dimensionally  $\beta$-absolute  winning for every 
$0 < \beta < 1/3$ (equivalently, for arbitrary small positive
$\beta$). Observe that the strongest case, $k=0$, is precisely
McMullen's absolute winning property.  
We will be mostly interested in the weakest case, $k = d-1$, in other words, the case when Alice
removes neighborhoods of affine hyperplanes. For the sake of brevity, $(d-1)$-dimensionally absolute  winning sets will be called {\sl hyperplane absolute winning\/}, or {\sl HAW} sets.

The following proposition summarizes some  properties of sets winning 
in this new version of the game:

\begin{prop}\label{kdimproperties}
\begin{itemize}
\item[(a)] HAW %(and thus $k$-dimensional absolute winning for any $k \le d-1$) 
implies 
%$\alpha$-strong winning $\forall\,\alpha<1/2$, and hence 
$\alpha$-winning for all $\alpha < 1/2$; hence HAW sets have  Hausdorff dimension $d$. 
\item [(b)] The countable intersection of $k$-dimensionally  absolute  winning sets
is $k$-dimensionally  absolute  winning.
\item[(c)]   The image of a $k$-dimensionally  absolute  winning  set under a $C^1$ diffeomorphism of $\rd$ is  $k$-dimensionally  absolute  winning. 
\end{itemize}\end{prop}

Parts (a) and (b) are straightforward, and part (c) will be proved in the next section in the following stronger form:

 \begin{thm}
\label{stability} 
 Let $S \subset \rd$ be  $k$-dimensionally  absolute  winning, 
 $U \subset \rd$ open, and
 $f: U \to \rd$ 
%a homeomorphism that is also 
a $C^1$ nonsingular map. %\comdmitry{Isn't a diffeomorphism automatically a homeomorphism? so I removed this additional
%condition, and also the footnote, maybe we will resurrect it later. But in any case, I changed it into
%a local version, please check!} 
%and $K$ a $k$-dimensionally diffuse set. 
%\footnote{From the proof and the nature of the game it is clear that the hypotheses can be weakened. Namely, $f$ can be allowed to not be a diffeomorphism on some set, as long as that set's intersection with any ball is contained in a finite union of $\kappa$-codimension subspaces which we can easily avoid. This comes into play with the map $(x,y,z) \mapsto (x^3,y^3,z^3)$ which has a singularity at the origin, or when gluing together maps defined on separate parts of $\rd$, when the gluing is along $\kappa$-codimension subspaces. Also, the derivative of $f$ need not be continuous as long as it is bounded on balls.}, 
Then $f^{-1}(S) \cup U^c$ is
$k$-dimensionally absolute  winning. Consequently, any $k$-dimensionally absolute  winning set 
is strongly $C^1$ incompressible.
\end{thm}

To deduce the statement made in the title of the paper, it remains to establish the HAW  property for the set of badly approximable vectors:

\begin{thm} 
\label{BA}
${{\bf{BA}}}_d\subset \rd$ is  hyperplane absolute winning.
\end{thm}

We point out that ${{\bf{BA}}_d}$ is not $k$-dimensionally absolute
winning for $k < d-1$, since, 
%DK
as in the case $k = 0$, 
% Alice must avoid $\{0\} \times
%\mathbb{R}^{d-1}$, and it is clear that in the $k$-dimensional
%absolute game 
Bob can play in such a way that his balls are always
centered on 
%DK
$\{0\} \times
\mathbb{R}^{d-1}$.
%this hyperplane.  

Here is one more class of winning sets which we show to be HAW. 
Let $\T^d \df \rd/\Z^d$ be the $d$-dimensional torus, and let $\pi:\rd\to\T^d$ denote the natural projection.
Given a nonsingular semisimple 
matrix $R\in \GL_d(\Q) \cap M_d(\Z)$ and a point $y \in \T^d$, define the set $$
\tilde E(R,y) \df \left\{\x\in \rd: y\notin\overline{ \{\pi(R^k\x )
%\bmod \Z^d
: k\in\N\}}\right\}\,;$$
in other words, a lift to $\rd$ of the set of points of $\T^d$ whose
orbits under the endomorphism $f_R$ of $\T^d$ induced  
by $R$, 
$$f_R(x) \df \pi(R\x)\text{ where }\x\in \pi^{-1}(\x) \,,$$
 do not approach $y$. This is a 
 %DK
 %null 
 set which has Lebesgue measure zero when $f_R$ is ergodic and, in the special case
 $y\in \Q^d/\Z^d$, was proved to be winning by Dani \cite{D} (see
 \cite{BFK} for more general results). We have 

%\comdmitry{Should we come up with an example of $R$ such that $\tilde
%E(R,y)$ is not $k$-dimensionally absolute winning for $k < d-1$?
%Should we mention/prove that for some $R$s these sets are absolutely
%winning?} 
\begin{thm} 
\label{ery} For every $R$ as above and any $y\in\T^d$, 
%DK 
the set
$\tilde E(R,y)$ is hyperplane absolute winning. 
\end{thm}

We  prove the three theorems stated above in the next section, and
then in \S\ref{fr} study games played 
on proper subsets of $\rd$.
%$S$ is said to be {\sl $k$-dimensionally absolute  winning\/} if is
%$k$-dimensionally $\beta$-absolute  winning for all .
%If $K = \rd$, we will simply say that $S$ is $\kappa$-absolute winning. (
  
\section{Proof of Theorems \ref{stability}, \ref{BA} and \ref{ery}}

\label{stab}

%Furthermore, the following  result is true, and will be proved in \S\ref{proofs}:
%We now come to one of the main results of the paper, invariance of the new winning property under $C^1$ maps.

%Combined with Theorem \ref{stability} and Proposition \ref{cap}, to be proved in \S\ref{fr},
%Corollary \ref{incompr general}??,
%the
% two theorems above will  imply Theorem \ref{incomprdiffuse} as well as  its version with $\bold{BA}$ replaced with $\tilde E(R,y)$.
% {\sb Probably it will be good to mention other examples 
%of HAW sets here, e.g.\ the sets of nondense orbits of toral endomorphisms.} 
%The aforementioned corollary is based on , that is, on a
% technique allowing one to intersect HAW sets and their diffeomorphic images with hyperplane duffuse subsets of $\rd$. The latter theorem will be used to derive Theorem \ref{incomprdiffuse}. 
 %which is the subject of the next section. 
%where it will be shown that under some natural assumptions on $K$, 
%our more general results will be stated.
%\vfil\eject

%Theorem \ref{C1} is clearly a special case of the above result, with $U = K = \rd$. It is also clear that Theorem \ref{incomprdiffuse} is an immediate consequence of Theorem \ref{BA}, Proposition \ref{cap}, Theorem \ref{dim}, Proposition \ref{strongwinning} and the  hyperplane case of Theorem  \ref{stability}.

\begin{proof}[Proof of Theorem \ref{stability}]
The idea is to `pull back' the strategy via the map $f$. We first
describe this informally. Alice is playing Bob with target set
$f^{-1}(S) \cup U^c$ and parameter $\beta \in (0,1/3)$; let us call this Game
1. She will define an appropriate constant $\beta' \in (0,1/3)$, and
consider a Game 2, played with target set $S$ and parameter $\beta'$,
for which it is assumed she has a winning strategy. For each move
$B_i$ made by Bob in Game 1, 
%DK she 
Alice 
will use $f$ to construct a set $B'_i$ which is
a legal move in Game 2. Using her strategy, she will have a move
$A'_i$ in Game 2, and she will use $f^{-1}$ to construct from this set
a legal move $A_i$ in Game 1. The fact that she is assured to win in
Game 2 will be shown to imply that this will also result in her
victory in Game 1. We now proceed to the details. 

No matter what $\beta$ is, note that if the diameters of the balls
$B_i$ do not tend to 
zero, then $\cap B_i$ has nonempty interior. Since $S$ is winning, it
is dense, so $f^{-1}(S)\cup U^c$ is dense as well, and $\cap B_i$ must
intersect it nontrivially. Thus, we may assume the diameters of $B_i$ 
do tend to zero. Furthermore, 
we may assume that there exists some $i \in \N$ such that $B_i \subset
U$,
since otherwise $B_i \cap U^c$ is a sequence of nonempty, nested closed subsets
of the closed set $U^c$, so that $\cap_i B_i \cap U^c$ is nonempty. 

\ignore{ is illustrated in these diagrams of the relevant sets and their radii:
\begin{equation} \begin{CD}
B'_{jk} @<{\text{dummy moves}}<< W'_{k} @<<{\mbox{\large{$\supset$}}}< f(A_k) \\
@V{f^{-1}}V{\bigcap}V @. @A{f}A{||}A \\
B_k @>>> \text{\footnotesize{Alice's move}}@>>> A_k
\end{CD} \end{equation}
\\ 
\begin{equation} \begin{CD}
\beta'^{nk}\rho' @<{\text{dummy moves}}<< \beta'^{nk+1}\rho'
@<<{\mbox{\large{$\supset$}}}< C\beta^{k+1}\rho' \\ 
@V{f^{-1}}V{\bigcap}V @. @A{f}A{\bigcup}A \\
\beta^k\rho  @>>> \text{\footnotesize{Alice's move}}@>>> \beta^{k+1}\rho
\end{CD} \end{equation}
 \\ \\
}

We will need to control the distortion of the map $f$. Denote by
$J_{\z} \df d|_{\z}f$ the derivative of $f$ at $\z$. First we note
that for any closed ball 
%DK
$B \subset U$ and any $\x, \y\in B$, 
we have 
$$
\|f(\x) - f(\y)\| \leq \sup_{\z \in B} \|J_{\z}\|_{op} \, \|\x - \y\|,
$$
where $\|J_{\z}\|_{op}$ denotes the operator norm of $J_{\z}$. For any
$\x_0 \in B$, denote by 
%DK
$$L_{\x_0}(\x) \df
f(\x_0)+J_{\x_0}(\x-\x_0)$$ the linear approximation to $f$ at $\x_0$,
and similarly let $\bar{L}_{\y'}$ denote the linear approximation to
$f^{-1}$ at $\y' \in f(B)$. We claim that 
%Since
%$f$ is nonsingular and $C^1$, the quantity
\eq{supratio}{
%\sup_{\x, \x_0 \in B} \frac{\|L_{\x_0}(\x)- f(\x)\|}{\|\x-\x_0\|},
%\ \ \ \  \  
\sup_{\y, \y_0 \in f(B)} \frac{\|\bar{L}_{\y_0}(\y)- f^{-1}(\y)\|}{\|\y-\y_0\|}
}
tends to zero 
%DK with 
when the diameter of $B$ tends to zero.
To see this, define $H_{\y_0}(\z) \df f^{-1}(\z) -
\bar{L}_{\y_0}(\z)$, so that $H_{\y_0}$ is differentiable with $d_{\z}H_{\y_0} =
J_{\z}^{-1} - J_{\y_0}^{-1}$. Since $f$ is $C^1$ and nonsingular,
$\sup_{\z,\y_0 \in B} \|d_{\z}H_{\y_0}\|_{op}$ tends to zero with the
diameter of $B$. But 
$$
\frac{\|\bar{L}_{\y_0}(\y)- f^{-1}(\y)\|}{\|\y-\y_0\|} =
\frac{\|H_{\y_0}(\y_0)-H_{\y_0}(\y)\|}{\|\y-\y_0\|} \leq \sup_{\z \in B}\|d_{\z}H_{\y_0}\|_{op}
\,,
$$
%DK and this 
which proves the claim. 

%To this end, we
%consider for a closed ball $B \subset \rd$ the quantity 
%\eq{supratio}{\sup_{\x,\y \in B}
%  \frac{\|f(\y)-f(\x)\|}{\|\y-\x\|}\,,} 
%i.e. the Lipschitz constant of $f|_B$.  
%Since $f$ is continuously differentiable, we can write
%$f(\y)=f(\x)+J_\x(\y-\x)+\delta(\y,\x)$, where $J_\x\df(Df)(\x)$ is
%the Jacobian matrix of $f$ at $\x$, and $\|\delta(\y,\x)\|$ is
%$o(\|\y-\x\|)$. Thus, when $\y$ is close to $\x$,
%$\frac{\|f(\y)-f(\x)\|}{\|\y-\x\|}$ is bounded from above by twice the
%operator norm $\|J_\x\|_{op}$ of $J_\x$.  
% and $C_1$ is a constant which Alice can force arbitrarily close to
% 1 For concreteness, we simply take $C_1=2$. 
Let $B_1$ be the initial ball chosen by Bob, and put 
$$C_1 \df \sup_{\x \in B_1} \|J_\x\|_{op}, \  \ C_2 \df \sup_{\x \in f(B_1)}
\|(J_\x)^{-1}\|_{op}.$$

%Let $\|J\| = \|J\|_{B_1}$ and $\|J^{-1}\| = \|J^{-1}\|_{B_1}$, 
%where $B_1$ is the initial ball chosen by Bob
%(after sufficiently many ``dummy moves'' by Alice) and let
Let $n \in \mathbb{N}$ be large enough that 
\eq{choose n}{
C(\beta + 1)\beta^{n-2} < 1, \ \ \ \mathrm{where \ } C \df 2C_1 \,
C_2,
}
%We use primes to denote the aspects of the $f(S)$ game. 
and let $\beta' = \beta^n < 1/3$.
By making enough dummy moves in the beginning of the game so
that Bob's ball $B$ becomes small enough, Alice can force the
quantity in \equ{supratio} to be bounded by $2C_2
(1+1/\beta)\beta'$. Moreover, as discussed above, after sufficiently
many dummy moves we may assume Bob's ball $B$ is contained in $U$.
We renumber balls so that such a sufficiently small ball is labelled
$B_1$.

\ignore{
Let $B_0$ be Bob's ball of radius $\rho'$ at the start of the game,
after Alice has made as many initial dummy moves as necessary to
ensure the above bounds as well as the radius of the balls being
smaller than $\rho_0$. Take the image of this ball under $f$; by the
established bounds, the image is contained in a ball of radius
$2C_1\rho'$, which we abbreviate by $\rho$, with this ball playing the
role of Bob's first move in the $S$-game.  
}
\ignore{Note that we can assume the diameter of $B_k$ tendsto zero, since
otherwise $\cap B_k$ contains an open set which must intersect with the dense set $f(S)$.
}

From this point on we will call our game (played with target set
$f^{-1}(S) \cup U^c$ and parameter $\beta$) Game 1, and use induction
to construct a legal sequence of moves 
$$B'_1 \supset B'_1 \ssm A'_1 \supset B'_2 \supset \dots$$
in a Game 2, played with target set $S$ and parameter $\beta'$. In
Game 2 Alice will play according 
to the strategy we have assumed she has. At the same time we will
describe a strategy for Alice in Game 1 and a sequence $1=j_1 < j_2 <
\cdots$, such that  
\eq{eq: first dict}{
f(B_{j_i}) \subset B'_i}
for all $i$.
This will imply that
$$f\big(\bigcap_i B_i\big) = f\big(\bigcap_i B_{j_i}\big) = \bigcap_i
f(B_{j_i}) \subset \bigcap_i B'_i \subset S,$$ 
so that $\bigcap_i B_i \in f^{-1}(S) \cup U^c$ and Alice wins. 

In
order to make this construction possible, in our induction we will
have to guarantee that some additional properties hold. 
Namely we will choose $j_1=1< j_2 < \cdots$, and sets
$A_i, B'_i$, $A'_i$ so that together with Bob's chosen moves $B_i$, we
will have 
%\eq{eq: sequence}{B_1 \supset B_1 \ssm A_1 \supset B_2 \supset \cdots}
%and
%\eq{eq: sequence'}{
%B'_1 \supset B'_1 \ssm A'_1 \supset B'_2 \supset \cdots 
%}
%and the following will hold 
for each $\ell$: 
\begin{itemize}
\item[(i)]
All sets chosen %in the sequences \equ{eq: sequence} and \equ{eq:
                %sequence'} 
are valid moves 
for Games 1 and 2 respectively. 
\item[(ii)]
Writing 
%\eq{eq: B'_i}{
$$B_{j_i} = B(\x_i, \rho_i), \ 
B'_i = B\big(f(\x_i), \rho'_i\big), \ \rho'_i =
C_1\rho_i,$$
%}
we have 
\eq{radiiratio}{\beta^n \le \frac{\rho_{i}}{\rho_{i-1}} < \beta^{n-1}
  \ \ \ \ \ (n \ \mathrm{as \ in \ \equ{choose n}}) .}
\item[(iii)]
$A'_i = \cl_i^{\prime(\varepsilon'_i)}$ are dictated by
Alice's strategy for Game 2, and 
$$
A_{j_i} \supset f^{-1}\left(\cl_i^{\prime(\eta \rho'_i)}\right),
\ \ \mathrm{where} \ \eta = (1+1/\beta)\beta'.$$
\end{itemize}

Note that \equ{eq: first dict} follows from (ii) and the definition of
$C_1$, thus specifying
choices satisfying (i)--(iii) will prove the theorem. 

Suppose we have carried out these choices up to stage $\ell-1$. 
Alice will play dummy moves until the first time the radius $\rho$ of Bob's
ball $B_m$ satisfies $\rho/\rho_{\ell-1} < \beta^{n-1}$, and set
$j_{\ell}=m$. Note that $j_{\ell}$ is well-defined since the radii 
chosen by Bob tend to zero, and at each step $\rho$ decreases by a
factor of at most $\beta$. Our choice of $j_{\ell}$ ensures that
\equ{radiiratio} holds for $i=\ell$.    
Now we let $B'_{\ell} = B\big(f(\x_\ell), \rho'_\ell\big)$.
We claim
that $B'_{\ell}$ is a valid move in Game 2.
By \equ{radiiratio} and the definition of $\rho'_i$,
\eq{radiiratio2}{\beta' = \beta^n \le \frac{\rho_\ell}{\rho_{\ell-1}}
  = \frac{\rho'_\ell}{\rho'_{\ell-1}} < \beta^{n-1} = \frac{\beta'}{\beta}.}
Since $B_{j_\ell} \subset B_{j_{\ell-1}}$, we must have $\dist(\x_{\ell-1},\x_\ell) \le 
	\rho_{\ell-1}-\rho_\ell$
and thus
$$\dist\big(f(\x_{\ell-1}), f(\x_\ell)\big) \leq
C_1\dist(\x_{\ell-1},\x_\ell)\le C_1(\rho_{\ell-1}-\rho_\ell) =
\rho'_{\ell-1}-\rho'_\ell\,.$$ 
Hence  $B'_\ell \subset B'_{\ell-1}$.
Since $B_{j_\ell} \cap A_{j_{\ell-1}} = \varnothing$ 
and 
$A_{j_{\ell-1}} \supset
f^{-1}\left(\cl_{\ell-1}^{\prime(\eta\rho'_{\ell-1})}\right)$,
it follows via \equ{radiiratio2} that 
$$
\dist(f(\x_\ell), \cl'_{\ell-1}) \geq \eta \rho'_{\ell -1} =
\left(1+\frac{1}{\beta}\right) \beta' \rho'_{\ell-1} >
\beta'\rho'_{\ell-1} + \rho'_{\ell}.
$$
That is, 
%$f(\x_{\ell}) \notin \cl_{\ell-1}^{\prime(\eta\rho'_{\ell-1})}$. 
%From \equ{radiiratio2} we obtain that $\rho'_\ell <
%\frac{\beta'}{\beta}\rho'_{\ell-1}$, and this implies that 
$B'_\ell
\cap \cl_{\ell-1}^{\prime (\beta' \rho'_{\ell-1})} = \varnothing,$ hence $B'_\ell
\subset B'_{\ell-1}\ssm A'_{\ell-1}$, proving the claim. 

Now choose $A'_\ell$ according to Alice's strategy in Game 2,
say $A'_\ell = \cl_{\ell}^{\prime(\varepsilon'_\ell)}$ for some
$\varepsilon'_\ell \le \beta'\rho'_\ell$. We will show that
$f^{-1}$ does not move $A'_\ell$ too much from a hyperplane
neighborhood in $B_{j_\ell}$. 
To this end, 
%determine the distortion
%on $f^{-1}\left(\cl_\ell^{\prime(\eta\rho'_\ell)}\right)$, 
fix $\y'\in \cl'_\ell$, 
%let $\bar{L}_{\y'}$ be the linear approximation to
%$f^{-1}$ at $\y'$, 
and define $\cl_{\ell} \df
\bar{L}_{\y'}(\cl'_\ell)$, which is also a
$k$-dimensional subspace. For any $\y \in B'_\ell$ with $\dist(\y,
\cl'_\ell) < \eta \rho'_\ell, $ let $\y_0$ be the projection of $\y $
onto $\cl'_\ell$. Then, by the choice of the initial ball $B_1$, we
have 
$$\dist\left(f^{-1}(\y), \bar{L}_{\y'}(\y_0)\right ) \leq
\dist\left(f^{-1}(\y), f^{-1}(\y_0)
\right) +  \dist\left(f^{-1}(\y_0), \bar{L}_{\y'}(\y_0)
\right) \leq 2C_2 \eta \rho_\ell'.$$

%Let 
%$\y \in \cl_\ell^{\prime(\eta \rho'_\ell)}$ be arbitrary, and let $\x$
%be its projection onto $\cl'_\ell$. We can
%write 
%%$$f^{-1}(\y)=f^{-1}(\x)+J^{-1}_{\x}(\y-\x)+\delta(\y,\x),$$ 
%%and  
%$$f^{-1}(\x)=f^{-1}(\y_0)+J^{-1}_{\y_0}(\x-\y_0)+\delta(\x,\y_0)\,.$$ 
% and
%hence  
%$$f^{-1}(\y)-[f^{-1}(\y_0)+J^{-1}_{\y_0}(\x-\y_0)]=f^{-1}(\y)-f^{-1}(\x)+\delta(\x,\y_0).$$
%We obtain {\comment{here we need an additional argument, see
%  Lindenstrauss}}
%$$\dist\big(f^{-1}(\y),\cl_{\ell}\big) \le 2C_2\dist(\x,\y)\,.$$
It follows that for 
$$\varepsilon_\ell \df 2C_2 \eta \rho'_\ell=C \eta \rho_\ell$$
we have 
$$A_{j_\ell}
\df \cl_{\ell}^{(\varepsilon_\ell)} \supset
f^{-1}\big(\cl_\ell^{\prime(\eta\rho'_\ell)}\big).$$
But, by \equ{choose n},
$C\eta < \beta$, so that 
$\varepsilon_\ell < \beta\rho_\ell$, i.e. $A_{j_{\ell}}$ is a
valid move for Alice in Game 1 satisfying (iii).
This concludes the inductive step  and completes the proof.
\end{proof}

For the proof of Theorem \ref{BA} we need the so-called `simplex
lemma', whose proof goes back to Schmidt and Davenport: 

\begin{lem} \cite[Lemma 4]{KTV}
\label{simplex}
\noindent 
For every $\beta \in (0,1)$ %let $R=\beta ^{\frac{-d}{d+1}}$ 
and for every $k\in \mathbb{N}$ let
\begin{center}
$U_k\df\left\{\frac{\p}{q}: q\in \mathbb{N},\ 
\p \in \mathbb{Z}^{d}\;\text{and}\;\; \beta ^{\frac{-d}{d+1}(k-1)}\leq q<\beta ^{\frac{-d}{d+1}k}\right\}$.
\end{center} 
Denote by $V_d$ the volume of the $d$-dimensional unit ball.
Then for every % $r>0$ such that 
\begin{equation}
\label{lemma} 0 < r< \beta (d!\, V_d)^{-1/d}
\end{equation}
and for every $\x\in\rd $ there exists an affine hyperplane $\mathcal{L}$ such that 
\begin{center}$U_k\cap B(\x,\beta^{k-1}r) \subset \mathcal{L}$.
\end{center}
\end{lem}
%For a proof see \cite[Lemma 4]{KTV}.

\begin{proof}[Proof of Theorem \ref{BA}] Let $\beta<1/3$. When the
  hyperplane $\beta$-absolute game begins, Alice makes dummy moves
  until the radius $\rho$ is small enough to satisfy  \eqref{lemma}
  with $r\leq \rho$. 
Then one  sets $c=\beta^2\rho$. After this, let $B_{j_k}$ be the
subsequence of moves where the radius $\rho_{j_k}$ first satisfies
$\beta^{k-1}\rho \geq \rho_{j_k} > \beta^k\rho$. On turns not in this
subsequence Alice makes dummy moves; on the turns $j_k$, consider the
rational points in $U_k$. $U_k \cap B_{j_k}$ is contained in a
hyperplane $\mathcal{L}_k$ by Lemma \ref{simplex}, so Alice chooses 
$$A_{j_k + 1}\df \mathcal{L}_k^{(\beta^{k+1}\rho)}\,.$$
% wants to avoid these rationals by $c\beta^{k-1}$, by the window in which the denominators lie. 
But note that
$$ \beta^{k+1}\rho =c\beta^{k-1} \ge c q^{-(1 + 1/d)}$$
whenever $q$ is a denominator of one of the rational points from
$U_k$. Thus any $\x\in B_{j_k + 1}$ 
is at least $c q^{-(1 + 1/d)}$ away from $\vp/q\in U_k$. Satisfying
this for all $k$ and comparing with \equ{def ba} shows $\cap_j B_j \in
{\bf{BA}}_d$. 
\end{proof}

\begin{proof}[Proof of Theorem \ref{ery}]
Let $\lambda$ be the %maximum absolute value of the eigenvalues 
spectral radius of $R$.
If $\lambda = 1$, then obviously
every eigenvalue of $R$ must have modulus $1$.
By a theorem of Kronecker \cite{Kro}, they must be roots of unity,
so there exists an $N\in\N$ such that the only eigenvalue of $R^N$ is $1$. 
Thus $R^N= I$.
Hence, for any $y \in\T^d$, $$\tilde E(R^N,y) \supset \R^d \smallsetminus (\y+\Z^d)\,,$$
where $\y$ is an arbitrary vector in $\pi^{-1}(y)$. Thus $\tilde E(R^N,y)$
is HAW, since $\y +\Z^d$ is countable.
Hence $\tilde E(R^N, z)$ is HAW whenever $z\in f_R^{-i}(y)$, where $0 \leq i < N$.
Thus the intersection $$\tilde E(R,y) = \bigcap_{i=0}^{N-1}\bigcap_{z\in f_R^{-i}(y)} \tilde E(R^N, x)$$
is also HAW.
%, by Corollary \ref{cap}.

Otherwise let $\ell \in\N$ be the smallest integer such that $\lambda^{-\ell} < \beta$,
and let $a = |\det (R)^{-\ell}|$.
Then $R^{-j}(\Z^d)\subset a\Z^d$ for $j\in\{0,1,\dots,\ell\}$.
Let 
$b > 0$ be such that
$R^{-j}\big(B(0,1)\big)\subset B(0,b)$ for $0 \leq j \leq \ell$.

Let $V$ and $W$ be the largest $R$-invariant subspaces on which all eigenvalues
have absolute value equal to, and less than, $\lambda$ respectively.
Then, since $R$ is semisimple, $\R^d = V \oplus W$. 
Since the eigenvalues of $R |_V$ are of absolute value $\lambda$
and $R$ is semisimple, there exists $\delta_1 > 0$ such that, for all $\vv\in V$ and $j\in \N$,
\begin{equation}
\label{delta_1}
\|R^{-j}\vv\| \leq \delta_1\lambda^{-j}\|\vv\|.
\end{equation}
Similarly, since all eigenvalues of $R^{-1}$ have absolute value at least $\lambda^{-1}$
and $R$ is semisimple, there exists $\delta_2 > 0$ such that for all $\vu\in \R^d$ and $j\in\N$ 
\begin{equation}
\label{delta_2}
\|R^{-j}\vu\| \geq \delta_2\lambda^{-j}\|\vu\|.
\end{equation}

Again, choose  an arbitrary vector $\y\in\pi^{-1}(y)$. Let $t_0$ be the minimum positive value of $\frac{1}{3b}\dist\big(\y - R^{-j}(\y),a \mathbf{z}\big)$,
ranging over $j \in \{0,1,\dots,\ell\}$ and $\mathbf{z} \in \Z^d$.
\ignore{
$$t = t(j_1,j_2,m_1,m_2) = \frac{\left\|\left(R^{-j_1}(y)+am_1\right) 
-\left(R^{-j_2}(y)+am_2\right)\right\|}{b}
$$
}
Then since $b\geq 1$, by the triangle inequality we have that,
for any $\vm_1, \vm_2 \in \Z^d$ and $0 \le j \le \ell$ such that $\y+\vm_1 \neq \R^{-j}(\y+\vm_2)$,
\begin{equation}
\label{distance}
\dist\left(B(\y+\vm_1,t_0),
R^{-j}(B(\y+\vm_2,t_0)\right) 
\geq \dist\left(B(\y+\vm_1, bt_0), B(R^{-j}(\y+\vm_2),bt_0)\right)
\geq t_0
\end{equation}
\ignore{
\begin{equation}
\label{distance}
\dist\left(R^{-j_1}(\y) + B\left(0,bt_0\right)+a\Z^d,
R^{-j_2}(\y)+B\left(0,bt_0\right)+a\Z^d\right) 
\geq t_0
\end{equation}
whenever $R^{-j_1}(\y)\ne R^{-j_2}(\y)$.
\ignore{Let $t_0$ be the minimum of $\frac13t(j_1,j_2,m_1,m_2)$, 
ranging over all $(j_1,m_1)$ and $(j_2,m_2)$ with $t(j_1,j_2,m_1,m_2) \neq 0$.
}
}
Let $k, j_1, j_2\in\Z_+$ be such that $j_1 \leq j_2$ and 
$\beta^{-k} \leq \lambda^{j_i} <\beta^{-(k+1)}$.
Note that, by our choice of $l$, $0 \le j_2-j_1 \leq \ell$.
By (\ref{delta_2}) and (\ref{distance}),
%and (\ref{distance})
for any $0 < t < t_0$ and $\vm_1,\vm_2 \in \Z^d$ such that 
$\y+\vm_1 \neq \R^{-j_2+j_1}(\y+\vm_2)$,
\begin{equation}
\label{distance2}
\dist\left(R^{-j_1}\big(B(\y+\vm_1,t)\big),R^{-j_2}\big(B(\y+\vm_2,t)\big)\right)
%\geq \delta_2 \lambda^{-lk} t_0 
\geq \delta_2t_0\beta^{k+1}.
\end{equation}

Let $p : \rn \to V$ be the projection onto $V$ parallel to $W$, and let
 $M = \|p\|_{op}$ (with respect to the Euclidean norm).
Alice will play arbitrarily until Bob chooses a ball of radius
$\rho \leq \frac{\beta}{2}\delta_2t_0$. 
Reindexing, call this ball $B_1$.
Choose a subsequence of moves $B_{jk}$ such that $B_{jk}$ has radius smaller than $\frac12\delta_2t_0\beta^{k+1}$,
so by (\ref{distance2}) 
%it intersects nontrivially with at most one 
if it intersects two sets of the form
$R^{-j_i}\big(B(\y+\vm_i,t)\big)$
 with 
$\beta^{-k} \leq \lambda^{j_i} < \beta^{-(k+1)}$, $j_1 \le j_2$, and $\vm_i\in\Z^d$,
then we must have $\y+\vm_1 = R^{-j_2+j_1}(\y+\vm_2)$.
Then by our choice of $b$,
$$R^{-j_2}\big(B(\y+\vm_2,t)\big) \subset R^{-j_1}\left(B\big(R^{-j_2+j_1}(\y+\vm_2),bt\big)\right) = 
	R^{-j_1}\big(B(\y+\vm_1,bt)\big)\,.$$
Thus all sets of the above form that intersect $B_{jk}$ must be contained in a single
set of the form $R^{-j}\big(B(\y+\vm,bt)\big)$.
Since
$\diam\big(p(B(\y,bt))\big) \leq 2Mbt$, we have by (\ref{delta_1}) that
for $t = \min\left(t_0,\frac{\beta\rho}{2Mb\delta_1}\right)$,
$$\diam\left(p\big(R^{-j}(B(\y+\vm,t))\big)\right)
%=\text{diam}(\pi(R^{-j}(B(y,t))))
\leq 2\delta_1Mbt\lambda^{-j} \leq 2\delta_1Mbt\beta^{k} \leq \beta^{k+1}\rho,$$
so each $R^{-j}\big(B(\y+\vm,t)\big)$ intersecting $B_{jk}$ 
is contained in $\mathcal{L}^{(\varepsilon)}$,
where $\mathcal{L}$ contains a translate of $W$ and $\varepsilon \leq \beta^{k+1}\rho$.

\ignore{
By our choice of $\beta$, 
\begin{equation}
\label{decaying1}
\mu\left(B(\mathcal{B}_k) \cap \mathcal{L}^{(2\varepsilon)}\right) < D\mu(B(\mathcal{B}_{k})) ,
\end{equation}
and by Definition \ref{decaying} (ii) 
\begin{equation}
\label{decaying2}
\mu(B(x_k,(1-\alpha)\beta^k\rho)) > D\mu(B(\mathcal{B}_k)),
\end{equation}
where $x_k$ is the center of $B(\mathcal{B}_k)$.

By (\ref{decaying1}) and (\ref{decaying2}), 
there exists a point $x\in\text{ supp }\mu \cap B(\mathcal{B}_k)$ such that, for any $m\in\Z^d$,
$$\dist(x,R^{-j}(B(y,t)+m)) > \alpha\beta^k\rho\,\,\,\, \text{and}\,\,\,\,
\dist(x,\partial B(\mathcal{B}_{k+1})) \geq \alpha\beta^k\rho.$$
Alice will choose this point as the center of $\mathcal{W}_{k+1}$.
}

On her $(k+1)$-st move, Alice will remove $\mathcal{L}^{(\varepsilon)}$.
By induction, Alice will play in such a way that $\x \in \cap B_{jk}$
satisfies $\|R^j(\x)-\vm-\y\| \geq t$ for all $\vm\in\Z^d$ and $j \in \Z_+$.
Hence, $\x \in \tilde E(R,y)$, and Alice wins.
\end{proof} 

\ignore{For this example, one can state a version of the theorem
  intrinsic to the setting of the torus.  
\comdmitry{I am not sure if we need this remark and corollary at
  all.}Define $E(R,y)=\{ x \in \T^d : y \notin \overline{\{R^j(x): j
  \in \mathbb{N} \}}\}$ and 
$$E_A=\cap_{y \in A} \cap{R} \pi(\tilde E(R,y)),$$ the intersection
being over a countable set $A$ and surjective semisimple
endomorphisms. \cite{D} showed $E_{\Q^d/\Z^d}$ is winning on $\T^d $,
and \cite{BFK} showed $E_A$ is (strong) winning on $\T^d \cap K$ for
an absolutely decaying $K$, which also follows from the proof
above. Using the fact that $\pi^{-1}(E_A)$ is $(d-1)$-dimensionally
absolute winning, we obtain results on its stability under maps of the
torus (rather than $\rd$), namely 
\begin{cor}
If $f_n$ is any sequence of $C^1$-diffeomorphisms of $\T^d$, $\cap_n
f_n(E_A)$ is winning on $\T^d \cap K$, where $K$ is absolutely
decaying in the sense that $\pi^{-1}(K)$ is. 
\end{cor}
\begin{proof} For a single map $f$, consider a small enough ball $B$
  in $\T^d$ and consider the lift $\pi_0^{-1}$ to a particular
  fundamental domain $\tilde B$ in $\R^d$. By the above proof $E_A$ is
  HAW on $\tilde B$. The map $g$ taking $\tilde B$ back to the torus,
  applying $f$, and taking a lift $\pi_1^{-1}$ to $\rn$ again is a
  diffeomorphism since $f$ is. Thus $g(E_A)$ is HAW, hence winning, on
  $g(\tilde B)$ by our main theorem, hence on $g(\tilde B) \cap
  \pi_1^{-1}(K)$. Projecting again gives the desired winning set in
  $\T^d$. 
\end{proof}}

\section{Games played on proper subsets}

\label{fr}
An application of Schmidt games to proving  abundance of badly
approximable vectors on certain fractal  
subsets of $\rd$ first appeared in \cite{F}, and was motivated by
earlier related results \cite{W, KLW, KW1, KTV}. 
The main idea is simple: if $K$ is a closed subset of $\rd$ and $0 <
\alpha,\beta < 1$, one lets Bob and Alice successively choose nested
closed balls in $\rd$ as in \equ{balls} and \equ{radii}, but with an  
additional constraint that the centers of all the balls belong to $K$. (An equivalent approach 
%\comdmitry{Not sure if we need this remark.}
is to view
$K$ as a metric space with the metric induced from $\rd$ and play the
$(\alpha,\beta)$-game there with $S\cap K$ being the target set, but 
replacing the usual containment of balls with the stronger one, namely
the containment of the corresponding balls in $\rd$.)  
One says that $S \subset \rd$ is  {\sl $(\alpha,\beta)$-winning on
  $K$\/} if Alice has a strategy guaranteeing that $\cap_iB_i\in S$
regardless of the way Bob chooses to play; 
{\sl $\alpha$-winning on $K$\/}  and {\sl winning on $K$\/} are defined accordingly.
Note that $\cap_iB_i$ is also automatically in $K$ since the latter
is closed and the centers of balls 
are chosen to be in $K$.

The following lemma, proved in   \cite{F} and adapted from Schmidt \cite{S1}, gives a condition on 
a set $K$ allowing one to estimate from below the \hd\ of $S\cap K$
for every $S$ which is winning on $K$. We will use the following
notation: for $K \subset \rd$, $\x \in K$, $\rho > 0$  and $0 < \beta
< 1$, let $N_K(\beta,\x,\rho)$ denote the maximum number of disjoint
balls of radius $\beta \rho$ centered on $K$ and  contained in
$B(\x,\rho)$.  

%We will use the following lemma proved in \cite{F}, adapted from Schmidt (\cite{S1}). 
\begin{lem}\label{liorstheorem} \cite[Theorem 3.1]{F}
Suppose there exists positive $M, \delta, \rho_0$ and $\beta_0$ such that \eq{lotsofballs}{N_K(\beta,\x,\rho) \geq M\beta^{-\delta}} whenever $\x\in K$, $\rho < \rho_0$ and $\beta < \beta_0$. Then $\dim(S\cap K\cap U) \ge \delta$ whenever $U$ is an open set with $U\cap K \ne\varnothing$ and $S$ is  winning  on $K$.
\end{lem}

We remark that in \cite{F} the lemma is stated for 
sets $K$ supporting absolutely friendly measures and with $\beta_0 = 1$ and $U = \rd$, but the proof does not require all this.
In particular, whenever \equ{lotsofballs} can be satisfied with $\delta = \dim(K)$, such as when $K$ supports
an Ahlfors regular measure (see the next section for more detail), this gives the full \hd\ 
of the intersection of $K$ and $S$ at any point of $K$. 

Our next goal is to define the $k$-dimensional $\beta$-absolute game played on $K$.
In this new version Bob and Alice will choose sets $B_i$ and $A_i$ as in the 
%$k$-dimensionally $\beta$-absolute 
game played on $\rd$, with Bob's choices
subjected to the additional constraint that they are centered on $K$.
That is, Bob initially chooses $\x_1\in K$ and $\rho_1> 0$, thus
defining a closed ball $B_1 = B(\x_1,\rho_1)$; 
then at each stage, after Bob chooses $\x_i \in K$ and $\rho_i > 0$,
Alice chooses an affine subspace $\cl$ of dimension $k$ and  $0 <
\varepsilon \leq \beta$, and then Bob chooses $\x_{i+1}\in K$ 
and $\rho_{i+1} \ge \beta\rho_i$ such that
$$B_{i+1}\df B(\x_{i+1},\rho_{i+1}) \subset B_i \ssm A_i\,.$$
As before,  we declare Alice the winner if \equ{Alicewins} holds,
% (in
%this case $\bigcap_{i=1}^\infty B_i$ is necessarily contained in $K$),
and  
say that $S\subset \rd$ is {\sl  $k$-dimensionally  $\beta$-absolute
  winning on $K$\/} if Alice has a  
winning strategy  
regardless of how Bob chooses to play. Observe however that it might
sometimes happen that Alice removes all or most  
of $K$ after finitely many steps\footnote{Note that the same scenario
  could happen in the game on $\rd$ if $\beta$ was chosen to be bigger
  than $1/3$, hence this restriction on $\beta$ in the 
  %DK
  rule
  introduced by McMullen.}, leaving no valid moves for Bob. 
For example, if $K$ is a singleton, Alice can remove it on her first turn, forcing
the game to stop. In this case, we will declare Bob the winner (and so
it would not be beneficial for Alice 
to play in this fashion).
%As a result, one would need certain conditions
%on $K$ guaranteeing that Alice can win even if the whole $\rd$ is the target set.

Let us say that $S\subset \rd$ is {\sl $k$-dimensionally absolute  winning on $K$\/}
if there exists $\beta_0$ such that Alice has a strategy to win  the
$k$-dimensional $\beta$-absolute game game on $K$ for all 
 $\beta \leq \beta_0$.
%(Our convention on \emph{intersecting} $S$ is in agreement with McMullen's and simplifies our proofs, but is slightly different from Schmidt's definition of general games in \cite{S1}, which requires $\subset S$.) $S$ is said to be $\kappa$-absolute winning on $K$ if is $(\kappa,\beta)$-absolute winning whenever $\beta \leq \beta_0$. If $K = \rd$, we will simply say that $S$ is $\kappa$-absolute winning. (Observe that the strongest case, $\kappa=d$, is McMullen's absolute winning.)
%With some abuse of notation, we say a set $S \subset \rd$ is $\kappa$-absolute winning
%on $K$ if $K \cap S$ is.
As before, $(d-1)$-dimensional absolute  winning on $K$ will be referred to as {\sl hyperplane absolute winning\/} ({\sl HAW}) on $K$.

\ignore{
$k$-dimensionally  $\beta$-absolute  winning implies $k$-dimensionally  $\beta'$-absolute  winning whenever $\beta' \ge \beta$.
We will say that $S$ is  {\sl $k$-dimensionally  absolute  winning\/}
if it is $k$-dimensionally  $\beta$-absolute  winning for every  
$0 < \beta < 1/3$ (equivalently, for arbitrary small positive
$\beta$). f, for example, our strategy at one point calls for removing
a set which leaves Bob only with moves having relative radius \beta_0
or smaller? In the \beta_0 game he can make a move and we can keep
playing according to our strategy and eventually win, but in the \beta
> \beta_0 game he has no moves and we lose.

Let us 
}

Because Alice loses when she leaves Bob with no valid moves, 
%it will generally be easier to describe a strategy for games played on subsets where this situation never arises. 
it will be useful to have a condition on $K$
which guarantees that
Bob will  have an available move even when Alice decides to choose a subspace which meets $K$.
The following geometric condition ensures 
exactly this for small enough initial radius,
which as we will see can be assumed without loss of generality.

%In order to prove incompressibility of $\bf BA$ and other  $K \subset \rd$, we define a game played on $K$ in which nested balls centered on $K$
%are chosen to be disjoint from certain neighborhoods of affine subspaces. In order to ensure
%that the game can be played, we need a geometric condition on $K$.
\begin{defn}
\label{diffuseness}
A closed set $K\subset\rd$ is said to be {\sl $k$-dimensionally $\beta$-diffuse\/} (here $0 \leq k < d$, $0<\beta<1$) if there exists 
$\rho_K > 0$ such that for any $0 < \rho \le \rho_K$, $\x\in K$,
and any $k$-dimensional affine subspace $\cl$, there exists $\x' \in K$ such that
$$
\x' \in B(\x,\rho)\ssm \cl^{(\beta\rho)}.
$$
We say that $K$ is 
{\sl $k$-dimensionally diffuse\/} if it is $k$-dimensionally $\beta_0$-diffuse for some $\beta_0 < 1$
(and hence for all $\beta \le \beta_0$). When $k = d-1$, this property will be referred to as {\sl hyperplane diffuseness}; clearly it implies $k$-dimensional diffuseness for all $k$.
\end{defn}

For a class of trivial examples of sets satisfying those
conditions, it is clear that $\R^d$ itself is
hyperplane $\beta$-diffuse for all $\beta < 1$, 
and so is the closure of any bounded open subset of $\R^d$ with
smooth boundary; more generally, any $m$-dimensional compact 
smooth submanifold of $\R^d$ is $k$-dimensionally diffuse whenever $m >
k.$ 
Moreover, many proper subsets of $\rd$ can also be shown to 
have the diffuseness property, such as, most notably, limit sets of
irreducible family of self-similar or self-conformal  
contractions of $\rd$ \cite{H, KLW, U2}. More examples will be given in \S\ref{meas}.
\ignore{The next example is also elementary, but it will be quite important in what follows:

%Another consequence of this new caveat is that, unlike games played on $\rd$, 
%being winning on $K$ with parameter $\beta$ might not imply winning
%with parameter $\beta'$ if $\beta'  > \beta$.\comdmitry{Is it
%possible to describe an example of this implication not going
%through, or it is just a precautionary measure?}  
Indeed,  for example Alice's strategy in the $\beta$-game at one point
may call for removing a set which leaves Bob only  
with allowed moves having relative  radius not bigger than $\beta$;
thus increasing the parameter to $\beta'$ 
may render her $\beta$-winning strategy unusable. In the $\beta$-game
Bob could make a move and Alice could keep playing according to her
strategy and eventually win, but in the $\beta'$ game Bob has no moves
and Alice loses. 

The above remark motivates the following definition:

\begin{lem}
\label{closed}\comdmitry{It will be nice to have a written proof of this lemma.}
Let $U\subset \rd$ be a bounded open set with smooth boundary, and let $K$ be  $k$-dimensionally $\beta$-diffuse with $U\cap K \ne\varnothing$. Then $\overline{U}\cap K$ is 
also  $k$-dimensionally  $\beta$-diffuse. In particular, $\overline{U}$ itself is  hyperplane $\beta$-diffuse for any $0<\beta<1/3$.
% (in particular, any set that is $\kappa$-absolute winning on $\rd$). 
\end{lem}
}

As we will be using this definition to choose balls in the absolute game,
it will be useful to have the following equivalent version of the definition.

\begin{lem}
\label{equiv diffuseness}
Let $K$ be $k$-dimensionally $\beta'$-diffuse
and let $\beta \le \frac{\beta'}{2+\beta'}$.
Then for any $0 < \rho \le \rho_K$, $\x\in K$, and any $k$-dimensional affine subspace $\cl$, 
there exists $\x' \in K$ such that
\eq{def diff}{B(\x',\beta\rho) \subset B(\x,\rho)\ssm \cl^{(\beta\rho)}\,.}
\end{lem}

\begin{proof} Clearly it suffices to prove the result for $\beta= \frac{\beta'}{2+\beta'}$.
%Then $\frac{2\beta}{1-\beta} = \beta'$.
Let $\x\in K$, $0 < \rho \le\rho_K$, and $\cl$ a $k$-dimensional affine subspace.
Since $\beta' = \frac{2\beta}{1-\beta}$,
we can use diffuseness to find a point $\x' \in K$ such that
$$\x' \in B\big(\x,(1-\beta)\rho\big) \ssm \cl^{(\beta'(1-\beta)\rho)}
	= B(\x,(1-\beta)\rho)\ssm\cl^{(2\beta\rho)}.$$
This $\x'$ satisfies \equ{def diff}.
\end{proof}

The diffuseness condition can also be stated 
in terms of microsets, a notion introduced by H. Fursternberg in \cite{Fu}.
%Let us first recall some definitions.
Let $B_1$ be the unit ball
in $\rd$ and given a ball $B$, let $T_B$ be the homothety sending $B$ to $B_1$.
Let $K$ be a 
%DK compact 
closed subset of $\rd$. 
%Then a set of the form $T_B(B\cap K)$, where $B$ is centered on $K$,
%is called a {\sl{miniset of $K$}}, and  
Any Hausdorff-metric limit point
of a sequence of sets $T_{B_i}(B_i\cap K)$, with each $B_i$ centered on $K$ and
$\diam(B_i)\to 0$, is called a {\sl{microset}} of $K$.

\begin{lem}
$K$ is $k$-dimensionally diffuse if and only if no 
microset of $K$ is contained in a $k$-dimensional affine subspace. 
\end{lem}

\begin{proof}
We will define the {\sl{$k$-dimensional width}} of a set $A$ to be
$$
\inf \{\varepsilon : \cl^{(\varepsilon)} \supset A \text{ for some $k$-dimensional affine subspace }\cl\}.
$$
Notice that $k$-dimensional width is continuous with respect to the
Hausdorff metric, and that $A$ is contained in a  $k$-dimensional
affine subspace if and only if its $k$-dimensional width is zero. 

First suppose $K$ is $k$-dimensionally $\beta$-diffuse for some $\beta$, and let
$B_i$ be a sequence of balls centered on $K$ with radii $\rho_i\to0$.
Then for sufficiently large $i$, the radius of $B_i$ is less than $\rho_K$,
so the diffuseness assumption guarantees that, for every
$k$-dimensional affine subspace $\cl$, 
$(B_i \ssm \cl^{(\beta\rho_i)}) \cap K \neq \varnothing$.
Thus, $T_{B_i}(B_i\cap K)$ has $k$-dimensional width at least $\beta$ for all
sufficiently large $i$, so any Hausdorff-metric limit point of this sequence does as well.

Conversely, suppose $K$ is not $k$-dimensionally $\beta$-diffuse.
Then there exist balls $B_i$ centered on $K$ with radii $\rho_i \to 0$
and $k$-dimensional affine subspaces $\cl_i$
such that $B_i \ssm \cl_i^{(\beta\rho_i)}$ is disjoint from $K$,
so that each $T_B(B_i\cap K)$ has $k$-dimensional width at most $\beta$.
By the continuity of 
%DK
the $k$-dimensional width, there exists
a microset of $K$ with $k$-dimensional width at most $\beta$.
Thus, if $K$ is not $k$-dimensionally diffuse, there exist microsets of $K$
with arbitrarily small $k$-dimensional width, and by compactness,
there are microsets with zero $k$-dimensional width.
\end{proof}

One advantage of playing on a diffuse set is the following lemma, which is a generalization
of \eqref{comparison}:

\begin{lem}
\label{small enough}
If $K$ is $k$-dimensionally $\beta$-diffuse and $S$ is
$k$-dimensionally absolute winning on $K$, then
$S$ is $k$-dimensionally $\beta'$-absolute winning on $K$ whenever
$\beta' \le \frac{\beta}{2+\beta}$. 
\end{lem}
\begin{proof}
There is no loss of generality in assuming that the diameters of the
balls chosen by Bob tend to zero. Thus when the $\beta'$-game begins,
Alice can choose $A_i$ disjoint from $B_i$ 
until Bob has chosen a ball of radius less than $\rho_K$.
Reindexing, call this $B_1$. We have assumed $S$ is $k$-dimensionally
$\beta''$-absolute winning for all small enough positive $\beta''$; in
particular for some $\beta'' \leq \beta'$.   
Note that at any stage of the game, playing with the parameter $\beta'$ instead of
$\beta''$ affords Alice more possible moves but eliminates some of Bob's possible moves.
By Lemma \ref{equiv diffuseness}, in view of  $\beta' \le \frac{\beta}{2+\beta}$, Bob 
will never win by having {\textsl{all}} 
of his potential moves removed,
so a $\beta''$-strategy for the initial ball $B_1$ is also a $\beta'$-strategy 
for this initial choice.
\end{proof}

Our next result asserts that diffuseness of $K$ is sufficient for a
lower estimate on the dimension of 
sets winning on $K$:

\begin{thm}
\label{dim}
Let $K\subset \rd$ be $k$-dimensionally $\beta$-diffuse set, let $S\subset\rd$ be winning on $K$,
and let $U$ be an open set with $U\cap K \ne\varnothing$.
% (in particular, any set that is $\kappa$-absolute winning on $\rd$). 
Then 
\eq{dimest}{\dim(S\cap K\cap U) \ge \frac{-\log(k+2)}{\log\beta-\log(2+\beta)}\,.}
%where $\beta'$ is as in Lemma \ref{equiv diffuseness}
\end{thm}

\begin{proof}%[Proof of Theorem \ref{dim}]
\ignore{Let $\rho_K$ be as in Definition \ref{diffuseness} and
let $\rho \le \rho_K$ be such that there exists
a ball centered on $K$ contained in $B$.
\ignore{
Let $\rho_0$ be the radius of $B$
and let $\rho_1 = \min(\rho_0,\rho_K)$,
where $\rho_K$ is as in Definition \ref{diffuseness}.}
It is easily seen that, since $K$ is $\kappa$-diffuse, 
each ball $B(x_0,\rho)$ with $x_0\in K$ and $\rho < \rho_K$ contains at least
$d-\kappa+2$ disjoint balls of radius $\beta\rho$ centered on $K$.
%; clearly, one of these can be taken to be $B(x_0,\beta\rho)$. 
Thus, by distributing mass in the natural way, we can construct a tree-like subset
of $K \cap B$ which supports a measure $\mu$
satisfying, for every $x \in \supp \mu$ and $\rho < \rho_K$,
$$C_1 \rho^\gamma \leq \mu(B(x,\rho)) \le C_2 \rho^\gamma,$$
where $\gamma = \frac{\log\left(\frac{1}{d-\kappa+2}\right)}{\log \beta}$ and $C_i > 0$.
Since $S$ is $\kappa$-absolute winning, it is
winning on $\supp \mu$, so by \cite[Proposition 5.1]{KW2} the proposition follows.}

It is easy to see that, by Lemma \ref{equiv diffuseness}, $N_K(\beta',\x,\rho) \geq k+2$ for 
$\beta'\le\frac{\beta}{2+\beta}$,
%\comdmitry{I think a bit more detail would be appropriate here. For example, what is $M$?}
$\rho < \rho_K$ and $\x\in K$; thus $N_K\left(\left(\frac{\beta}{2+\beta}\right)^n,\x,\rho\right) 
\geq (k+2)^n$, so we get the hypothesis of Lemma \ref{liorstheorem}  with $\delta=\frac{-\log\left(k+2\right)}{\log\beta - \log(2+\beta)}$ and $M = \frac{1}{k+2}\,$. \end{proof}

\ignore{
\begin{lem}
For any $k$-dimensionally $\beta_0$-diffuse set $K$,
if $\beta < \beta' \leq \beta_0$ and 
$S$ is $k$-dimensionally $\beta$-absolute winning on $K$, 
then $S$ is also $k$-dimensionally $\beta'$-absolute winning on $K$.
In particular, $S$ is $k$-dimensionally absolute winning on $K$
if there exists some $\beta_1 > 0$ such that 
$S$ is $k$-dimensionally $\beta$-absolute winning on $K$ for all
$\beta < \beta_1$.
\end{lem}

\begin{proof}
We will describe a strategy for the $k$-dimensional $\beta'$-absolute game on $K$.
The game begins with Bob choosing $B_1$ of radius $\rho_1 \leq \rho_K$.
Since $\beta < \beta'$, Alice can choose $A_1$ according to her $\beta$-strategy.
Again since $\beta <\beta'$, Bob's next choice $B_2$ will be a valid move
in the $\beta$-game, so once again Alice can use her strategy for that game to choose $A_2$.
Continuing in this way, Alice can ensure that $\cap B_i$ intersects $S$.
\end{proof}
}

Proving a set to be $k$-dimensionally  absolute  winning on $K$ has several useful implications.
First of all, it implies the winning property discussed in the beginning of the section:
%We now illustrate the relative strength of our definition:

\begin{prop}
\label{strongwinning}
Let $K\subset \rd$ be  $k$-dimensionally $\beta$-diffuse, and let $S\subset \rd$
be $k$-dimensionally absolute  winning on $K$. 
Then $S$ is $\frac{\beta}{2+\beta}$-winning\footnote{In fact 
the argument shows that $S$ is  $\frac{\beta}{2+\beta}$-strong winning  on $K$, 
where the latter is defined
as in \cite{Mc}.} on $K$.
\end{prop}
\begin{proof}
%\comdmitry{Can anybody check the proof after my editorial changes? I downgraded it to the usual winning, as Barak suggested.}
%\comliorryan{Looks good.}
%Let $K$ be $(\kappa,\beta_0)$-diffuse and assume $S$ is 
%$\kappa$-absolute winning.
By Lemma \ref{small enough}, we know that $S$ is $k$-dimensionally 
$\beta'$-absolute
winning on $K$ for any $\beta' \le \frac{\beta}{2+\beta}$.
Let $\alpha = \frac{\beta}{2+\beta}$, $0 < \beta'' < 1$, and 
$\beta'= \alpha\beta'' < \frac{\beta}{2+\beta}$.
We want to win the $(\alpha,\beta'')$-game on $K$ using the strategy
we have in the $k$-dimensional $\beta'$-absolute game.
When the game begins, Alice will choose $A_i$ to be concentric with $B_i$
until, reindexing, Bob chooses $B_1=B(x,\rho)$ with $\rho <\rho_K$.
Now at the $i$-th stage, suppose Alice's strategy is to remove $\cl^{(\varepsilon\rho)}$ where $\varepsilon \leq \beta'< \alpha$.
Since $K$ is $k$-dimensionally $\beta$-diffuse, by Lemma \ref{equiv diffuseness} there exists
$\x' \in B(\x,\rho)\cap K$ such that
$B(\x',\alpha\rho) \subset B(\x,\rho) \ssm\cl^{(\varepsilon\rho)}$, so she chooses $\x'$ as the center of her ball.
\ignore{
We have \[ \mu(B(x,(1-\alpha)\rho) \cap H^{(\varepsilon+\alpha)\rho})
\leq C\frac{(\beta'+\alpha)^\gamma}{(1-\alpha)^\gamma}.\] Clearly,
since $\beta<\alpha$, for small enough $\alpha$, the right-hand side
is $<1$ regardless of $\beta'$.  
When this is the case there exists a point in $B(x,(1-\alpha)\rho)
\ssm H^{(\beta+\alpha)\rho}$, and this point is a valid move for Alice
in the $(\alpha,\beta')$ game, which succeeds in avoiding
$H^{\beta\rho}$. 
}
Bob's next move $B_i$ is a ball contained in Alice's of radius at
least $\alpha\beta''\rho=\beta'\rho$, so this is a valid move for Bob
in the $k$-dimensional $\beta'$-absolute game. 
Since $B_i$ is the same ball in the two games, we have $\cap B_i \in S$,
so $S$ is $\alpha$-winning.
\end{proof}

In particular, in view of Theorem \ref{dim},  \equ{dimest} holds
whenever $K\subset \rd$ is  $k$-dimensionally $\beta$-diffuse  and
$S\subset \rd$ is $k$-dimensionally absolute winning on $K$. One also has

\begin{prop}
\label{cap} 
%Let $K\subset \rd$ be  $k$-dimensionally diffuse.
Let $K$ be any closed subset of $\rd$, $\beta > 0$, and for each $i\in \N$ let $S_i$
be $k$-dimensionally $\beta'$-absolute winning 
on $K$ for any $\beta' \le \beta$. Then the
 countable intersection $\cap_i S_i$  
is also $k$-dimensionally $\beta'$-absolute winning on $K$ for any $\beta' \le \beta$.  
\ignore{In particular, in view of Lemma \ref{small enough},
if $K$ is $k$-dimensionally diffuse, then the countable intersection of
sets $k$-dimensionally absolute winning on $K$ is
also $k$-dimensionally absolute winning on $K$.}
\end{prop}
\begin{proof} 
The proof is exactly the same as Schmidt's original proof
\cite[Theorem 2]{S1} for $\alpha$-winning sets.
%, except for our
%different convention of allowing $\cap B_i$ intersecting $S$ instead
%of $\cap B_i \subset S$, but obviously the proof still goes through. 
\end{proof}

Now that we have described consequences of being able to win an absolute game on a  diffuse set $K$,
the next natural question is how to one can verify this kind of winning  property. Remarkably, it turns out that this property can be simply extracted from the corresponding one for $K= \rd$.
%DK Namely
More generally, the following holds:

%It appears that using the above 
%The following inheritance property is one of the advantages of the $(\kappa,\beta)$-game.
\begin{prop}\name{prop: intersection}
\label{cap K}%\comdmitry{Perhaps a little more detailed proof will be better. For example, how is diffuseness used? how do $\beta_0$ and $\rho_K, \rho_L$ become important? should Alice make dummy moves for a while?}
%\comliorryan{Changed proof; please check.}
If $L\subset K$ are both $k$-dimensionally diffuse  and $S\subset
\rd$ is $k$-dimensionally  absolute  winning on $K$, 
then it is also $k$-dimensionally  absolute  winning on $L$.
In particular,
every set which is $k$-dimensionally  absolute  winning  on $\rd$ is $k$-dimensionally  absolute  winning on 
every  $k$-dimensionally diffuse set.
\end{prop}

\begin{proof} 
We assume $S$ is $k$-dimensionally  absolute  winning on $K$, and show that it is
$k$-dimensionally  absolute  winning on $L$. 
Let $\beta$ be small enough that $L$ and $K$ are both $k$-dimensionally $\beta$-diffuse,
so that by Lemma \ref{small enough} $S$ is $k$-dimensionally 
$\frac{\beta}{2+\beta}$-absolute winning
on $K$. Then consider the $\frac{\beta}{2+\beta}$-game played on $L$.
Alice will choose $A_i$ disjoint from $B_i$ until Bob has chosen a ball of
radius less than $\min(\rho_L, \rho_K)$. Reindexing, call this $B_1$.
Since playing on $L$ instead of $K$ restricts Bob's choices but not Alice's, and because
Lemma \ref{small enough} 
guarantees that Bob will never win by having no valid moves available to him,
Alice may use her strategy for the game played on $K$.
\end{proof}

Combining the above proposition with Theorem \ref{stability} we obtain
the following statement, proving Theorem \ref{incomprdiffuse}: 

 \begin{thm}
\label{stabilitydiffuse} 
 Let $S \subset \rd$ be  $k$-dimensionally  absolute  winning, 
 $U \subset \rd$ open,
%\comdmitry{Isn't a diffeomorphism automatically a homeomorphism? so I
%removed this additional 
%condition, and also the footnote, maybe we will resurrect it
%later. But in any case, I changed it into 
%a local version, please check!} 
and $K$ a $k$-dimensionally diffuse set. Then for any  $C^1$ nonsingular map $f: U \to \rd$, 
%a homeomorphism that is also 
%\footnote{From the proof and the nature of the game it is clear that
%the hypotheses can be weakened. Namely, $f$ can be allowed to not be
%a diffeomorphism on some set, as long as that set's intersection with
%any ball is contained in a finite union of $\kappa$-codimension
%subspaces which we can easily avoid. This comes into play with the
%map $(x,y,z) \mapsto (x^3,y^3,z^3)$ which has a singularity at the
%origin, or when gluing together maps defined on separate parts of
%$\rd$, when the gluing is along $\kappa$-codimension subspaces. Also,
%the derivative of $f$ need not be continuous as long as it is bounded
%on balls.},  
the set $f^{-1}(S) \cup U^c$ is
$k$-dimensionally absolute  winning on $K$. Consequently, for any
sequence $\{f_i\}$ of $C^1$ diffeomorphisms of $U$ onto (possibly
different) open subsets of $\rd$, the set 
\equ{intersdiffuse}
has  positive \hd.
\end{thm}

\begin{proof}
For the first part, in view of Proposition \ref{cap K}, it suffices to
show that $f^{-1}(S) \,\cup\, U^c$ is $k$-dimensionally 
absolute winning on $\rd$, which is exactly the conclusion of Theorem
\ref{stability}. 
As for the second assertion, the union of
the set
\equ{intersdiffuse} with $U^c$ is $k$-dimensionally absolute  winning on $K$ by the first part, Lemma \ref{small enough} and
Proposition \ref{cap}, hence winning on $K$ by Proposition \ref{strongwinning}, hence its intersection with $U$, that is, the set
\equ{intersdiffuse} itself, has positive \hd\ by  Theorem  \ref{dim}.\end{proof}

\ignore{One can now collect the results proved in the last two sections and conclude with the

\begin{proof}[Proof of Theorem \ref{incomprdiffuse}]
Indeed, in view of Proposition \ref{cap K}, it suffices to show that $f^{-1}(S) \,\cup\, U^c$ is $k$-dimensionally
absolute winning on $\rd$, which is exactly the conclusion of  Theorem \ref{stability}.\end{proof}

It is  clear now that Theorem \ref{incomprdiffuse} is an immediate consequence of Theorem \ref{BA}, Proposition \ref{cap}, Theorem \ref{dim}, Proposition \ref{strongwinning} and the  hyperplane case of Theorem  \ref{stabilitydiffuse}.}

\ignore{

\begin{cor} If $S \subset \rd$ is a $k$-dimensionally  absolute  winning  set and $f: \rd \to \rd$ is a %homeomorphism that is also a 
$C^1$-diffeomorphism, then $f(S)$ is $k$-dimensionally  absolute  winning   on any  $k$-dimensionally  diffuse  set $K$.
\end{cor}}

In particular, by Theorems  \ref{BA} and \ref{ery},  the sets $\BA_d$
and $\tilde E(R,y)$,  as well as their countable intersections and
differmorphic images, always intersect hyperplane diffuse subsets of
$\rd$, and the intersection has positive \hd. This conclusion not only 
generalizes many known results on intersection of sets of these types
with fractals, see \cite{BBFKW, BFK}, 
but provides a more conceptual `two-step' proof: to find uncountably
many points of $K\cap S$, one has to check separately the diffuseness
of $K$ and 
the absolute  winning property of $S$.
\smallskip

It is instructive to point out another simple consequence of
Proposition \ref{cap K}. Let $\Gamma$ be a one-cusp discrete group of
isometries of the hyperbolic space $\mathbb{H}^{n}$, and let
$D(\Gamma) $ be the set of lifts of endpoints of bounded geodesics in
$\mathbb{H}^{n}/\Gamma$. Dani \cite{Dani-rk1} showed that $D(\Gamma)$ is
winning, and McMullen \cite{Mc} strengthened this result, showing $D(\Gamma)$ is
absolute winning. Aravinda \cite{Ar}
%DK showed 
proved that the intersection of 
the $D(\Gamma)$ with any $C^1$ curve is winning (and hence has
\hd\ $1$). This can now be seen to be an 
immediate corollary of Proposition \ref{cap K} and
the $0$-dimensional 
diffuseness property of smooth curves.

Our next goal is to study the concept of the incompressibility on $K$, which calls for showing that the set  \equ{intersdiffuse}
has full \hd. %\ref{incompr},  \ref{incomprdiffuse}, as well as their version for other HAW sets like $\tilde E(R,y)$.
However for that we need additional assumptions on $K$,
phrased in terms of properties of a measure whose support is equal to $K$. This is discussed in the next section.

\ignore{Thus
under certain assumptions on the set 
 $K$, phrased in terms of a measure it supports,   $\bf BA$ %and other similarly defined sets are 
 is winning on $K$, and one has $\dim(S\cap K) > 0$ for any set $S$ which is winning on $K$;
 in fact in many cases the dimension of $S\cap K$ can be shown to be equal to $\dim(K)$.
 This theme was further exploited in \cite{BBFKW, BFK}. }

\section{Measures and the proof of Theorem \ref{incompr}}

\label{meas}

 %Our goal here is to show that under
 %some mild condition on the set $K$ the `absolute' game described at the end of the previous section
 %can be played on $K$, and that this condition is also sufficient for a lower estimate on the dimension of winning sets. 
% `absolute' modification of Schmidt's game and 
%A convenient way to phrase those conditions has been via measures supported on $K$. 
We start with  a
definition introduced  in \cite{KLW}: 
%The central theme of all those papers is the use of decay properties of measures $\mu$ for extracting information
%on \di\ properties of vectors in the support of $\mu$. To explain this motivation, let us recall a few definitions.
if $\mu$ is a locally finite
Borel measure on $\R^d$ and  $C, \gamma>0$,
one says that $\mu$ is
{\sl $(C,\gamma)$-absolutely decaying\footnote{This terminology differs slightly from the
one in \cite{KLW}, where a less uniform version was considered.}\/} if
there exists $\rho_{0} >0$ such that, 
for all $0 < \rho < \rho_0$, all $\x \in\text{supp}\,\mu$, all affine
%DK 
hyperplanes 
$\mathcal{L}\subset \R^d$ and all $\varepsilon > 0$, one
has 
\eq{ad}{\mu\left(B(\x,\rho)\cap \mathcal{L}^{(\varepsilon)}\right) <
  C\left(\frac{\varepsilon}{\rho}\right)^{\gamma}\mu\big(B(\x,\rho)\big)\,.} 
%Here and below $\mathcal{L}^{(\varepsilon)}$ stands for the closed $\varepsilon$-neighborhood of $\mathcal{L}$.
We will say that $\mu$ is {\sl absolutely decaying\/} if it is
$(C,\gamma)$-absolutely decaying for some positive  $C, \gamma$. 

Another useful property %, which often comes in a package
%\comdmitry{Need to reword, since this is taken directly from \cite{BFK}.}
%\comdima{How to say it better?} 
%with absolute decay, 
is the so-called 
%DK
Federer (doubling) condition. One says that  $\mu$ is
{\sl $D$-Federer\/} if there exists $\rho_{0} >0$ 
%and $D > 1$ 
such that
\eq{fed}{\mu\big(B(\x ,2\rho)\big) <
  D\mu\big(B(\x,\rho)\big)\,\quad\forall\,\x\in \supp \,
  \mu,\ \forall\,0 < \rho<\rho_0\,,} 
and {\sl Federer\/} if it is $D$-Federer for some $D>0$. 
Measures which are both absolutely decaying and Federer are called {\sl absolutely friendly\/}, 
a term coined in \cite{PV}.

Many examples of absolutely friendly measures can be found in
\cite{KLW, KW1, U2, SU}. The Federer 
%(also called doubling) 
condition is very well studied; it obviously holds when  $\mu$ is
{\sl Ahlfors regular\/}, i.e.\  when there exist positive $\delta,
c_1,c_2,\rho_0$ 
such that %for every $\x\in \supp \, \mu$ and $0 < \rho<\rho_0$ one has
\eq{pl}{c_1\rho ^{\delta}\leq\mu\big(B(\x,\rho)\big)\leq
  c_2\rho^{\delta}\,\quad\forall\,\x\in \supp \, \mu,\ \forall\,0 <
  \rho<\rho_0 \,.}  
The above property for a fixed $\delta$ will be referred to as  {\sl
  $\delta$-Ahlfors regularity\/}. 
It is easy to see that the \hd\ of the support of a $\delta$-Ahlfors
regular measure is equal to $\delta$. 
An important class of examples of absolutely decaying and  Ahlfors regular measures is provided by
limit measures of irreducible families of contracting self-similar \cite{KLW} or self-conformal \cite{U2} transformations of $\R^d$  satisfying the open set condition, as defined by Hutchinson \cite{H}. 
%\comdmitry{Maybe say some more about examples?}
See however \cite{KW1}
for an example of an absolutely friendly measure which is not Ahlfors regular. 

\ignore{Following a terminology introduced in \cite{KLW, PV}, given
  $C,\gamma > 0$ say that a locally finite Borel measure $\mu$ on
  $\rd$ is 
{\sl $(C,\gamma)$-absolutely   decaying\/} if there exists $\rho_{0} >0$ 
%and $D > 1$ 
such that
\eq{ad}{\begin{aligned}
\mu\big(B(\x,\rho)\cap \mathcal{L}^{(\varepsilon)}\big) 
	< C(\varepsilon/\rho)^{\gamma}\mu\big(B(\x,\rho)\big)\ \\
	 \text{ for any  affine hyperplane }\mathcal{L}\subset \rn\quad \ \\  \text{
and }\forall\,\x\in \supp \, \mu,\  0 < \rho<\rho_0,
	\  \varepsilon > 0\,.\ \ 
\end{aligned}}
Here $B(\x,\rho)$ stands for the closed Euclidean ball in $\R^d$ of radius $\rho$ centered at $\x$, 
and $\mathcal{L}^{(\varepsilon)} \df \{\x \in\rn : \dist(\x,\mathcal{L}) \leq \varepsilon\}$ is the closed $\varepsilon$-neighborhood of $\mathcal{L}$.

%We now exhibit a known class of sets which is an example of diffuse sets.

%\begin{defn}
%\label{decaying}
For $n\in\N$, let $\mu$ be a $\sigma$-finite
Borel measure on $\R^d$ and let $C, \gamma>0$.
We say that $\mu$ is
\underline{$(C,\gamma)$-absolutely decaying} if there exists $\rho_{0} >0$ such that,
for all $0 < \rho < \rho_0$ and all $\x \in\text{supp }\mu$, we have
that for each hyperplane $\mathcal{L}\subset \R^d$ and each
$\varepsilon > 0$, 
\begin{center}
$\mu\big(B(\x,\rho)\cap \mathcal{L}^{(\varepsilon)}\big) <
  C\left(\frac{\varepsilon}{\rho}\right)^{\gamma}\mu\big(B(\x,\rho)\big)$. 
\end{center} We call the support of such a measure an absolutely decaying set.
\end{defn}}

Our next result shows that supports of absolutely decaying measures have the diffuseness property:

%(For supports of power law measures the same proof works and recovers the usual dimension bound.)
%Note also that it was shown 
%While it was previously known by \cite{KW2} that sets which are winning on an absolutely {\bf{friendly}} set have positive dimension, this shows one can weaken the condition to absolutely decaying.
%\end{remark}

\begin{prop}
\label{diffuse}
Let $\mu$ be absolutely decaying; then  $K = \supp\,\mu$ is hyperplane diffuse (and hence also $k$-dimensionally diffuse for all $1 \le k < d$).
\end{prop}

\begin{proof}
Let $K=\supp \mu$. If $B=B(\x,\rho)$, where we assume $\rho<\rho_0$, and 
%$\cl_\kappa$ is a $(d-\kappa)$-dimensional affine subspace, then clearly there exists a hyperplane
if $\cl$ is an affine hyperplane, then 
%containing $\cl_\kappa$, so
\[ %\mu(B(x,(1-\beta)\rho) \cap \cl_\kappa^{(2\beta\rho)}) \leq 
\mu\left(B\big(\x,\rho\big) \cap \cl^{(\beta\rho)}\right)< C\beta^\gamma\mu\left(B\big(\x,\rho\big) \right).\] 
If $\beta = \left(\frac{1}{C}\right)^{1/\gamma}$, then
$C\beta^\gamma < 1$, hence %there exists a point in 
the intersection of $K$ with 
$B\big(\x,\rho\big) \smallsetminus \cl^{(\beta\rho)}$ is not empty.
\end{proof}

In particular, in view of Theorem \ref{dim}, if $K$ is the support of
$\mu$ as above, $U$ an open set with $K\cap U\ne \varnothing$  and $S$
is winning on $K$, then $S\cap K\cap U$ has positive \hd . Note though
%\comdmitry{I am confused by the relationship between this lemma and
%Theorem \ref{dim}; does one imply the other, in view of Proposition
%\ref{diffuse}? or does this lemma give a stronger bound?}  
that we can directly obtain a dimension bound for sets winning on
supports of \absd\ measures  using Lemma \ref{liorstheorem}: 

\begin{lem} \label{absd dim bound} 
Let $\mu$ be $(C,\gamma)$-absolutely decaying, let $K = \supp\,\mu$,
and let $S\subset \rd$ be winning on $K$. Then $\dim(S\cap K\cap U)
\geq \gamma$  for any open $U$  with $K\cap U\ne \varnothing$ . 
\end{lem}

\begin{proof}%[Proof of Lemma \ref{absd dim bound}]
For any $\beta<1$, let $k$ be such that for sufficiently small $\rho$,
$\inf_{\x\in K}N_K(\beta,\x,\rho) \geq k$. For any $\x \in K$, take
$k$ hyperplane neighborhoods of width $2\beta$ contained in the ball
around $\x$ of radius $(1-\beta)\rho$.  
%\comdmitry{The wording of the proof needs to be changed. Which disjoint balls? What holds for all $\x$? }
Their total measure is at most
$kC\frac{(2\beta)^\gamma}{(1-\beta)^\gamma}$ of the measure of the
ball $B\big(\x,(1-\beta)\rho\big)$. Thus if $k <
\frac{(1-\beta)^\gamma}{C(2\beta)^\gamma}$, there is a point of $K$
outside the union of these hyperplane neighborhoods and in
$B\big(\x,(1-\beta)\rho\big)$. This point can be the center of a new
$\beta\rho$-ball in $B(\x,\rho)$, and now we have $k+1$ disjoint balls
(since each hyperplane neighborhood contains a ball). Thus
$N_K(\beta,\x,\rho) \geq \frac{(1-\beta)^\gamma}{C(2\beta)^\gamma}$,
so for sufficiently small $\beta$ we get $N_K(\beta,\x,\rho) \geq
M\beta^{-\gamma}$ with some constant $M$. Now we apply Lemma
\ref{liorstheorem} which yields the desired dimension estimate. 
\end{proof}

We remark that the statement of the above lemma has been known, in
view of \cite[Proposition 5.2]{KW2}, for absolutely 
friendly measures, that is, under the additional Federer condition,
and the  lemma weakens the assumption 
to being just \absd.
Note also  that a similar argument was used in \cite{F} to prove 

%\comryan{I changed the statement of this lemma so it's no longer
%technically from Lior's paper. Is it true?} 
\begin{lem}  \label{pl dim bound} \cite[Theorem 5.1]{F} Let $\mu$ be
  $\delta$-Ahlfors regular, let $K = \supp\,\mu$, and let $S\subset
  \rd$ be winning on $K$. Then $\dim(S\cap K \cap U)  = \delta =
  \dim(K \cap U)$ for every open set $U \subset \rd$ with $V \cap K
  \neq \varnothing$. 
\end{lem}

In particular, the conclusion of the lemma holds whenever $S$ is
$k$-dimensionally absolute winning on $K$ as above. Combining with
Theorem \ref{stability}, 
we obtain

\begin{cor}
\label{incompr general}
Let $\mu$ be absolutely decaying and Ahlfors regular, and let
$S\subset \rd$ be $k$-dimensionally absolute winning; then   
$S$ is strongly $C^1$ incompressible on $\supp\,\mu$.
\end{cor}

Note that Theorem \ref{incompr} follows immediately from Corollary \ref{incompr general}
and Theorem \ref{BA}.
\smallskip

%\comdmitry{I did not really like this remark; it needs to be clarified. Are you trying to say that any hyperplane diffuse set has a subset supporting an a.d.\ measure? then why not have a proposition like that? also, the proof of Theorem 3.4 only has 3 lines, I don't see how in it something can be iteratively distributed. More details are needed. And we need to ask a question of whether any diffuse set supports an a.d.\ measure (if yes, the whole approach of the paper makes little sense)}

The following proposition, together with Proposition \ref{diffuse}, shows that the notions of  hyperplane diffuse sets and supports of absolutely decaying measures are closely related.

\begin{prop}
\label{decaying subset}
Let $K$ be a hyperplane diffuse subset of $\rd$ and $U \subset \rd$ open
with $U \cap K \neq \varnothing$.
Then there exists an absolutely decaying measure $\mu$ such that
$\supp\mu \subset K \cap U$.
\end{prop}

Whether or not it is possible to construct an absolutely decaying measure $\mu$ with
$\supp\mu = K$ is an open question.

\smallskip
In order to prove the proposition, we will need 
%DK a technical 
the following
%DK lemma.

\begin{lem}
\label{hyp lemma}
Given $\beta_0 > 0$ there exists a positive $ \beta' < \beta_0$ such that,
for every $\x \in \rd$, $\rho > 0$, and $\y_1, \dots, \y_d \in B(\x,\rho)$
such that the balls $B(\y_i,\beta_0\rho)$ are contained in $B(\x,\rho)$ and are pairwise disjoint,
if a hyperplane $\cl$ intersects each ball $B(\y_i,\beta'\rho)$, then
$$B(\x,\rho) \cap \cl^{(\beta'\rho)} \subset B(\x,\rho) \cap \cl(\y_1,\dots,\y_d)^{(\beta_0\rho)},$$
where $\cl(\y_1,\dots,\y_d)$ is the hyperplane passing through the points $\y_i$.
\end{lem}

\begin{proof}%[Proof of the lemma]
By rescaling, it suffices to prove the case that $\x=0$ and $\rho = 1$.
If the statement fails then there exists a sequence of $d$-tuples 
%DK changed the indexing here and in the next proof
$(\y_{1,n},\dots, \y_{d,n})$ as above and hyperplanes $\cl_n$ such that
$\cl_n$ intersects each $B(\y_{i,n},\frac1n)$ but
$$B(0,1) \cap \cl_n^{(\frac1n)} \not\subset B(0,1) \cap \cl(\y_{1,n},
	\dots,\y_{d,n})^{(\beta_0)}.$$
By the compactness of $B(0,1)$ there is a limit point $(\y_1,\dots,\y_d)$ of 
%DK the
that
sequence of $d$-tuples, and the corresponding subsequence of hyperplanes
$B(0,1) \cap \cl_{n_j}$ converges to $B(0,1) \cap \cl(\y_1,\dots, \y_d)$ in the Hausdorff metric,
as does $B(0,1) \cap \cl(\y_{1,n_j},\dots, \y_{d,n_j})$.
Hence, for each $\varepsilon > 0$, there is some $m > 1/\varepsilon$
such that $\dist(\y_{i,m},\y_i) < \varepsilon$ for each $1\le i \le d$ and
\eq{hyp contain}{B(0,1) \cap \cl_m^{(\varepsilon)} \subset B(0,1) \cap \cl(\y_1,\dots,\y_d)^{(2\varepsilon)} \subset B(0,1) \cap \cl(\y_{1,m},\dots,\y_{d,m})^{(3\varepsilon)}.}
But, by the definition of $\cl_m$,
$$B(0,1) \cap \cl_m^{(\varepsilon)}\not\subset 
		B(0,1) \cap \cl(\y_{1,m},\dots,\y_{d,m})^{(\beta_0)}.$$
This contradicts \equ{hyp contain} as 
%DK soon 
long as $3\varepsilon < \beta_0$.
\end{proof}

\begin{proof}[Proof of Proposition \ref{decaying subset}]
We will construct $\mu$ by iteratively distributing mass.
Let $\rho_K$ be as in Definition \ref{diffuseness} and
let $\beta_0 < 1/3$ be small enough to satisfy
\equ{def diff} with $\beta = \beta_0$.
Let $\beta \le \frac13\beta_0$ be small enough so that $\beta' = 2\beta$ is
as in Lemma \ref{hyp lemma}.
Let $\x_0 \in U \cap K$ and $0 < \rho_0 < \rho_K$ small enough to guarantee $B(\x_0,\rho_0) \subset U$.
Take $B(\x_0,\rho_0) = A_0$ and assign it mass $1$.
Now, given any $1 \leq \ell \leq d$ and points
$\y_1,\dots, \y_{\ell} \in A_0$, there is a hyperplane $\cl$ passing through these points,
and Lemma \ref{equiv diffuseness} guarantees the existence of 
a point $\y_{\ell +1}\in A_0 \cap K$ such that $B(\y_{\ell+1},\beta_0\rho_0) \subset A_0$
is disjoint from $\cl^{(\beta_0\rho_0)}$. Hence %, by finite induction,
we can choose $\x_{1,1}, \dots, \x_{1,d+1} \in K$ with 
$B(\x_{1,i},\beta_0\rho_0) \subset A_0$ disjoint
such that $B(\x_{1,d+1},\beta_0\rho_0)$ is disjoint from
$\cl(\x_{1,1}, \dots, \x_{1,d})^{(\beta_0\rho_0)}$.

By Lemma \ref{hyp lemma}, this implies that for any hyperplane $\cl$, its neighborhood 
$\cl^{(\beta\rho_0)}$ can intersect at most
$d$ of the balls $B(\x_{1,i},\beta\rho_0)$. Iindeed, if it intersects the first $d$ of these
balls, then $\cl$ intersects the corresponding balls $B(\x_{1,i},\beta'\rho_0)$,
which implies
$\cl^{(\beta\rho_0)} \subset \cl(\x_{1,1},\dots,\x_{1,d})^{(\beta_0\rho_0)}$ which is disjoint
from $B(\x_{1,d+1},\beta_0\rho)$)
We will take $A_1^{(i)} = B(\x_{1,i},\beta\rho_0)$ and distribute the mass uniformly. 
Note that these
children of $A_0$ are separated by $2\beta\rho_0$ from each other.
Now suppose $A_{j,i} = B(\x_{j,i},\beta^j\rho_0)$ comprise the $j$th stage of the construction,
with each $\x_{j,i} \in K$.
For each $i$, we again use the diffuseness condition and the lemma to find
$\x_{j+1,\ell} \in K$, for $(i-1)(d+1) < \ell \le i(d+1)$, such that
$A_{j+1,\ell} = B(\x_{j+1,\ell},\beta^{j+1}\rho_0) \subset A_{j,i} $ are separated by
at least $2\beta^{j+1}\rho_0$ and $\cl^{(\beta^{j+1}\rho_0)}$ intersects at most $d$ of these balls
for any hyperplane $\cl$, and distribute the mass uniformly among them. 
It is clear that the support of the limit measure $\mu$
is contained in $K$. We claim it is absolutely decaying.

Let $\rho < \rho_0$, $\x \in \supp \mu$, $0 < \varepsilon < \frac12\beta\rho$, 
and $\cl$ a hyperplane.
Take $i, j \in \N$ such that
\eq{rho range}{2\beta^i\rho_0 \leq \rho < 2\beta^{i-1}\rho_0}
and
\eq{epsilon range}{\frac12\beta^{j+1}\rho \leq \varepsilon < \frac12\beta^{j}\rho.}
Since $\x \in \supp \mu$, it is contained in some stage-$i$ ball $A_i$
and some stage-$(i-1)$ ball $A_{i-1}$.
By \equ{rho range}, $B(\x,\rho) \supset A_i$,
and since $A_{i-1}$ is separated by $2\beta^{i-1}\rho_0$ from every other stage-$(i-1)$
ball, $B(\x,\rho) \cap \supp \mu \subset A_{i-1}$.
Since $\varepsilon < \beta^{i}\rho_0$,
%and the children of $A_{i-1}$ are separated by at least $2\beta^{i}\rho_0$, 
%DK
the $\varepsilon$-neighborhood of $\cl$ intersects at most $d$ of the children of $A_{i-1}$. 
Since the measure is distributed uniformly among them, these
children carry $\frac{d}{d+1}$ of the measure of $A_{i-1}$. Similarly,
by \equ{epsilon range}, for each $\ell \leq j$, $\cl^{(\varepsilon)}$ intersects at most $d$ of the
children of each stage-$(i+j-2)$ ball. It thus follows that
\begin{align*}
\mu\big(B(\x,\rho\big) \cap \cl^{(\varepsilon)}) &\leq \left(\frac{d}{d+1}\right)^{j}\mu(A_{i-1})\\
	 & = \left(\frac{d}{d+1}\right)^{j}(d+1)\mu(A_{i}) \le \left(\frac{d}{d+1}\right)^{j}(d+1)\mu\big(B(\x,\rho)\big)\,.\\
\end{align*}
But, taking $\gamma = \frac{\log [d/(d+1)]}{\log \beta} > 0$, we have by \equ{epsilon range} that
$$\left(\frac{d}{d+1}\right)^j = \beta^{j\gamma} \le 2^\gamma\beta^{-\gamma}\left(\frac{\varepsilon}{\rho}\right)^\gamma.$$
Hence, \equ{ad} holds for $C = 2^\gamma\beta^{-\gamma}(d+1)$ whenever 
$\varepsilon < \frac12\beta\rho$. But if $\varepsilon \ge \frac12\beta\rho$,
then 
%DK
$C\left(\frac{\varepsilon}{\rho}\right)^\gamma$ is not less than $1$, therefore $\equ{ad}$ holds trivially.
Thus $\mu$ is absolutely decaying.
\end{proof}

\section{Concluding remarks}\label{remarks}

%There are quite a few questions that are left hanging. Perhaps we should collect those here, and then decide whether we should mention them, or answer some of them.

\subsection{Winning sets and strong incompressibility} 
The main result of the paper specifies a strengthening of the winning
property which enables one to deduce strong $C^1$ incompressibility,
and, in particular, which is invariant under $C^1$ diffeomorphisms. We
remark here that the class of sets winning in Schmidt's original 
version of the game does not have such strong invariance
properties. In other words, it is possible 
to  exhibit an example of a winning subset of $\R^2$ whose
diffeomorphic image is not winning.  
Here is one construction, motivated by \cite[\S 4]{Mc}. Given
$a \in \N$ and $\theta > 1$,  
we denote by $P_{a,\theta}$ the  {\sl almost arithmetic progression\/}
which starts with $a$ and has difference $\theta$, that is,  
$$
P_{a,\theta} \df \{\lceil a + j\theta\rceil : j = 0,1,\dots\}\,.
$$
Then consider
\eq{counterex}{S =\{(x,y) \in \R^2 : \exists\,a \in \N\text{ and
  }\theta > 1\text{ such that }x_i = y_i = 0\text{ whenever }i \in
  P_{a,\theta}\}\,,} 
where  $x_i,y_i$ are the digits of the base 3 expansions of $x,y$.
We also let $$\tilde S =\{(x,y) \in \R^2 : x_i = y_i = 0 \text{ for some }i \in \N\}\,.$$ 
Clearly $S$ is a subset of $\tilde S$; it is not hard to see that $S$
has Lebesgue measure zero, and $\tilde S$ has full Lebesgue measure. 

\ignore{. It would be interesting 
to find out whether or not it is possible to construct a winning
subset $S$ of $\rd$ which is  not strongly  incompressible. Or maybe
even for which  the intersection \equ{inters}, with no uniform bound
on bi-Lipschitz norms of $f_i$, can be empty. (Recall that it is
impossible for $d=1$, as proved by Schmidt in \cite{S1}.) 

However it is not hard to construct an example of a winning subset $S$ of $\R^2$ and a family of nonsingular linear
self-maps of $\R^2$ such that the intersection \equ{inters} is not winning.
\ignore{The following example shows that this stability property does not hold for\comdmitry{Following Barak's suggestion, I removed the concept of strong winning, so the exposition of this example should be modified.}
the classes of winning sets and strong winning sets. }
Indeed, let $$S =\{(x,y) \in \R^2 : x_i = y_i = 0 \text{ for some }i \in \N\}\,,$$ where  $x_i,y_i$ are the digits of the base 3 expansions of $x,y$. It is easily seen that $S$ is $1/9$-strong winning.}

\begin{prop} \label{example}\begin{itemize}
\item[(a)] $S$ (and therefore $\tilde S$) is winning. 
\item [(b)] Define $f_n(x,y) \df (3^{-n}x,y)$. Then $f_n(\tilde S)$ is not $\frac{4}{3^n}$-winning, and hence neither is $f_n(S)$.
\end{itemize}\end{prop}
\begin{proof}
For (a), take $\alpha = 1/108$ and, without loss of generality, assume the radius of $B_1$ is $\rho < \frac{1}{6\alpha}$. Then 
 let $a$ be the integer satisfying
$$\frac{1}{6\cdot 3^{a+1}} \leq \alpha\rho < \frac{1}{6\cdot 3^a}.$$
This implies $\rho \ge 18 \cdot 3^{-(a+1)} = 2 \cdot 3^{1-a}$.
Hence, $B_1$ contains a square of
sidelength $3^{-a}$ consisting of pairs $(x,y)$ with $x_a = y_a = 0$.
Since $\alpha\rho < \frac{1}{2\cdot 3^a}$, Alice can choose $A_1$ to be contained in this
square.

Now take $\theta = -\log_3 (\alpha\beta)$, and assume $A_k$ has been chosen so that
for all $(x,y) \in A_k$ and $0 \leq j <k$ we have $x_i = y_i = 0$ for $i = \lceil a + \theta j\rceil$. 
Then the radius of $B_{k+1}$ is
$$(\alpha\beta)^k\rho = 3^{-\theta k} \rho \geq 2\cdot 3^{1-a-k\theta}.$$
Let $i = \lceil a + k\theta\rceil$. Then $i \geq a + k\theta$, so
$$2\cdot 3^{1-i} \leq 2\cdot 3^{1-a-k\theta}\leq (\alpha\beta)^k\rho.$$
Thus, $B_k$ must contain a square of sidelength $3^{-i}$ consisting of
pairs $(x,y)$ with $x_i= y_i = 0$.
Furthermore, $i < a + k\theta + 1$, so
$$(\alpha\beta)^k\alpha\rho < \frac{1}{6\cdot 3^{a+k\theta}} \leq \frac{1}{2\cdot 3^{i}}\,.$$
Therefore Alice can choose $A_k$ to be contained in this square.
This implies that $(x,y) \in \cap A_k$ must satisfy $x_i = y_i = 0$ for all
$i \in P_{a,\theta}$. 

Next, we give a strategy for Bob with target set $f_n(\tilde S)$, with 
$\alpha = 4\cdot 3^{-n}$ and $\beta = \frac{1}{4}3^{-n}$.
We will show that Bob can play the game in such a way that, regardless of Alice's
play, we will have for each $i \in \N$ either $x_{i+n} = 1$ or $y_i = 1$.
Bob begins by choosing $B_0$ to be a ball of radius $1/2$ consisting of pairs $(x,y)$
with $x_0 = y_0 = 1$.
Note that the radius of $A_k$ is $\frac{\alpha(\alpha\beta)^{k-1}}{2} = \frac{2}{3^{2n(k-1)+n}}$,
so it contains a square of sidelength $3^{-2nk +n-1}$ consisting of
points $(x,y)$ satisfying $x_{2nk -n+1} = y_{2nk-n+1} = 1$. 
Also, the radius of $B_{k+1}$ is $\frac{(\alpha\beta)^k}{2} = \frac{1}{2\cdot 3^{2nk}}$,
so $B_{k+1}$ can be chosen so that $x_i = y_i = 1$
for $(2k - 1)n < i \le 2kn$. Let $(x,y) \in \cap B_k$ and $i \in \N$.
Let $m = 0,1,\dots$ be such that $mn < i \le (m+1)n$.
If $m$ is odd, then $m = 2k-1$ for some $k \in \N$ and $y_i = 1$.
If $m$ is even, then $m = 2(k-1)$ for some $k \in \N$,
therefore $(2k-1)n <  i +n \le 2kn$ and $x_{i+n} = 1$.
Hence, there does not exist an $i \in \N$ for which $x_{i+n} = y_i = 0$,
so $(x,y) \not\in f_n(S)$.
\end{proof}

The above  argument also shows that there exists a $C^\infty$
diffeomorphism $f$ of $\R^2$ (constructed so that for each $m$ there
exists a strip $n \leq x \leq n+1$ on which $f^{-1}=f_m^{-1}$) such
that $f(\tilde S)$ is not winning.  However, it is clear that $\tilde
S$ is of full Lebesgue measure, and hence so are all its diffeomorphic
images considered above and their countable intersections. Therefore
it cannot serve as a counterexample to the strong affine
incompressibility. Whether or not such a counterexample is furnished
by $S$ as in \equ{counterex}, that is, whether or not the intersection
\equ{inters} for this $S$ has  \hd\ less than $2$, is not clear to the
authors. 
 In fact it would be interesting  
to find out whether or not there exists a winning subset $S$ of $\rd$
which is  not strongly (strongly affinely, strongly $C^1$)
incompressible. Or maybe even for which  the intersection
\equ{inters}, with no uniform bound on bi-Lipschitz norms of $f_i$,
can be empty (this is impossible for $d=1$, as proved by Schmidt in
\cite{S1}). 
 %Other related questions are:
%\subsection{Incompressibility properties of winning sets} 

\subsection{VWA is strongly incompressible}\name{vwa}
A straightforward application of the Baire category theorem
shows that $\bigcap_i f_i^{-1}(S)$ is nonempty whenever $S$ is residual and the $f_i$ are
homeomorphisms; however this does not imply 
that it has full
%DK
Hausdorff dimension, and thus does not imply
incompressibility. 
%DK 
K.\ Falconer \cite{Falconer} introduced a theory which
implies lower bounds on the 
%DK Hausdorff 
dimension of $\bigcap_i
f^{-1}_i(S)$ for
certain residual sets $S$. These ideas were developed
further by A.\ Durand in \cite{durand}, 
%DK
see also references therein. We illustrate these results by
exhibiting another 
incompressible set arising in 
%DK
Diophantine approximation. Namely, for $\tau > 0$ denote by $J_{d,\tau}$ the set of {\sl $\tau$-approximable\/} vectors in $\rd$, that is, %those for which 
 $$J_{d,\tau}\df \left\{\x\in\rd : \text
{ there are infinitely many } \mathbf{p} \in \Z^d,\,
q \in \N\text{ with }\left\|\x -
  \frac{\mathbf{p}}q\right\| < \frac1{q^{\tau}}\right\}\,.$$It is well known that when $\tau > \frac{d+1}d$, the sets\  $J_{d,\tau}$
  has Lebesgue measure zero, and $\dim( J_{d,\tau})$ tends to $d$ as  $\tau \to \frac{d+1}d$.
  It was shown by Durand, see  \cite[Thm.\ 1 and Prop.\ 3]{durand}, that for any nonempty open $U\subset \R^d$
  and a sequence $f_1, f_2, \ldots$
of bi-Lipschitz maps $U \to \R^d$, the \hd\ of the intersection $\ \bigcap_i f_i^{-1}(J_{d,\tau})$ 
is equal to  $\dim( J_{d,\tau})$. Thus, if one denotes by $\mathbf{VWA}_d \df \bigcup_{\tau> \frac{d+1}{d}}J_{d,\tau}$ the set of {\sl very well approximable\/} vectors in $\R^d$, it follows that $$\dim\left( \bigcap_i f_i^{-1}(J_{d,\tau})\right) = d\,;$$ that is,  $\mathbf{VWA}_d $ is strongly incompressible.

Note however that  $\mathbf{VWA}_d $ is not a winning set, since it is contained in $\rd\ssm\BA_d$. 
Furthermore, the analogue of Theorem \ref{incompr} does not hold for $\mathbf{VWA}_d $: indeed, using Lemma \ref{simplex} it is not hard to find a closed set $K$,  supporting an absolutely decaying and Ahlfors regular measure, which is contained in   $\BA_d$ (see \cite{KW1} where such constructions are explained), so that $K\cap \mathbf{VWA}_d = \varnothing$.

%DK
The example described above brings up a  natural open question: is the set $\BA_d$ (together with
its cousins defined in terms of toral endomorphisms) strongly incompressible? that is, can one 
weaken the $C^1$ assumption on the maps $f_i$ to just bi-Lipschitz? 
%What if, in the spirit of \cite{Mc},  `bi-Lipschitz' is replaced by `quasi-conformal', with or without the uniform bound on 
The methods of the present paper do not seem to be enough to answer this question. 

\ignore{
%DK we say 
recall that $\x \in \R^d$ is called {\sl very well approximable\/} if there is some
$\varepsilon>0$ such that there are infinitely many solutions $\mathbf{p} \in \Z^d,
q \in \N$ to the inequality $\displaystyle{\left\|q\x -
  \mathbf{p}\right\| \leq \frac{1}{q^{1/d+\varepsilon}}}$. We denote the
set of very well approximable vectors in $\R^d$ by $\mathbf{VWA}_d$. We have:

\begin{prop}
%DK
\cite{durand}\name{VWA incompressible}
The set $\mathrm{VWA}_d$ is strongly incompressible. 
\end{prop}

\begin{proof}
%This follows easily from results of Durand. 
%Throughout this proof we
%use the results and 
%notations of \cite{durand}. 
Let $U\subset \R^d$ be open and let $f_1, f_2, \ldots$ be a sequence
of bi-Lipschitz maps $U \to \R^d$.  
In the notation of \cite{durand}, we have 
$$\mathrm{VWA}_d = \bigcup_{\tau> \frac{d+1}{d}}J_{d,\tau}.$$
For each $i$ and each $\tau > (d+1)/d$, by \cite[Prop. 3]{durand}, $J_{d,\tau} \in
G^{\mathrm{Id}^{(d+1)/\tau}}(f_i(U))$, and by \cite[Thm. 1]{durand}, this
implies that $\dim \, \bigcap_i f_i^{-1}(J_{d,\tau}) \geq
\frac{d+1}{\tau}.$ Hence 
$$\dim \, \bigcap_i f_i^{-1}(\mathrm{VWA}_d) \geq \frac{d+1}{\tau}
\longrightarrow_{\tau \to (d+1)/d} d.$$
\end{proof}
}
\medskip

{\bf Acknowledgements.} We are 
%DK 
thankful to Mike Hochman, Curt McMullen and Keith
Merrill for useful discussions, and to Arnaud Durand for pointing out the
relevance of \cite{durand}. We gratefully acknowledge the support
of the Binational Science Foundation, Israel Science Foundation, and
National Science Foundation through grants 2008454, 190/08, and DMS-0801064 respectively. 

\bibliographystyle{alpha}

\begin{thebibliography}{22}



\bibitem{Ar} C.S.\ Aravinda, \textsl{Bounded geodesics and Hausdorff
  dimension},   Math.\ Proc.\ Cambridge Philos.\ Soc.\  {\bf 116}
  (1994),  no.\ 3, 505--511. 

\bibitem{BBFKW}
R.\ Broderick, Y.\ Bugeaud, L.\ Fishman, D.\ Kleinbock and B.\ Weiss,
\textsl{Schmidt's game, fractals, and numbers normal to no base},
Math.\ Research Letters {\bf 17} (2010), 307--321. 



\bibitem{BFK} R.\ Broderick, L.\ Fishman and D.\ Kleinbock, \textsl{Schmidt's game, fractals, and orbits of toral endomorphisms}, Ergodic Theory Dynam.\ Systems, to appear. 



\bibitem{Dani-rk1} S.G.\ Dani, 
\textsl{Bounded orbits of flows on \hs
s}, 
Comment.\ Math.\ Helv.\ {\bf 61} (1986), 636--660. 

\bibitem{D}\bysame,  \textsl{On orbits of endomorphisms of tori and the Schmidt game}, 
Ergod.\ Th.\  Dynam.\ Syst.\ {\bf 8} (1988), 523--529.


\bibitem{Dani conference} \bysame, \textsl{On badly approximable
  numbers, Schmidt games and bounded orbits of flows}, in:   Number
  theory and dynamical systems (York, 1987),  69--86, London
  Math.\ Soc.\ Lecture Note Ser., {\bf 134}, Cambridge Univ.\ Press,
  Cambridge, 1989.  


\bibitem{durand} A. Durand, \textsl{Sets with large intersection and ubiquity}, 
 Math.\ Proc.\ Cambridge Philos.\ Soc.\  {\bf 144}  (2008),  no.\ 1, 119--144. 

\bibitem{ET} M.\ Einsiedler and J.\ Tseng, \textsl{Badly approximable
  systems of affine forms, fractals, and Schmidt games}, J.\ Reine
  Angew.\ Math., to appear. 


\bibitem{Falconer} K.\ Falconer,  \textsl{Sets with large intersection properties},
J.\ London Math.\ Soc.\ (2) {\bf 49} (1994), 267--280.


\bibitem{Fae} D.\  F\"arm, \textsl{Simultaneously Non-dense Orbits
  Under Different Expanding Maps},  
Dynamical Systems: An International Journal, {\bf 25} (2010), no.\ 4,  531--545.

%DK
\bibitem{FPS} D.\ F\"arm, T.\ Persson and J.\ Schmeling,
  \textsl{Dimension of Countable Intersections of Some Sets Arising in
    Expansions in Non-Integer Bases}, Fund.\ Math.\  {\bf 209} (2010), 157--176. 



\bibitem{F} L.\ Fishman, \textsl{Schmidt's game on fractals},  Israel J.\ Math.\
{\bf 171}  (2009), no.\ 1, 77--92.

\bibitem{F2} \bysame, \textsl{Schmidt's game, badly approximable matrices and fractals},  J.\ Number Theory 
{\bf 129} (2009), no.\ 9, 2133--2153.

\bibitem{Fu} H.\ Furstenberg, \textsl{Ergodic fractal measures and
  dimension conservation}, Ergodic Theory Dynam.\ Systems {\bf 28}
  (2008),  
no.\ 2, 405--422.


\bibitem{G} F.W.\ Gehring, \textsl{Topics in quasiconformal mappings},
Lecture Notes in Math., {\bf{1508}} (1992), 62--80.









\bibitem{H} J.E.\ Hutchinson, \textsl{Fractals and
self-similarity},  Indiana Univ.\ Math.\ J.\ {\bf 30} (1981), no.~5,
713--747.


\bibitem{KLW}D.\ Kleinbock, E.\ Lindenstrauss and B.\ Weiss, 
\textsl{On fractal measures and diophantine approximation}, Selecta Math.\ 
{\bf 10} (2004), 479--523.




\bibitem{KW1}D.\ Kleinbock and B.\ Weiss, 
\textsl{Badly approximable vectors on fractals}, 
Israel J.\ Math.\ {\bf 149} (2005), 137--170.

\bibitem{KW2}\bysame, 
\textsl{Modified Schmidt games and Diophantine approximation with weights}, 
Advances in Math.  {\bf 223} (2010), 1276--1298.



\bibitem{KW3}\bysame, 
\textsl{Modified Schmidt games and  a conjecture of Margulis}, Preprint, {\tt arXiv:1001.5017}.





\bibitem{KTV}S.\ Kristensen, R.\ Thorn and S.L.\ Velani, 
\textsl{Diophantine approximation and badly approximable sets}, 
Advances in Math.\ {\bf 203} (2006), 132--169.


\bibitem{Kro}L.\ Kronecker, \textsl{Zwei S\"atse 
\"uber Gleichungen mit ganzzahligen CoefÞcienten}, J.\ Reine Angew.\ Math.\ {\bf 53}
(1857), 173--175; see also Werke, Vol.\ 1, 103--108, Chelsea Publishing Co., New York, 1968.

\bibitem{Mc}C.\ McMullen, \textsl{Winning sets, quasiconformal maps and 
Diophantine approximation}, to appear in Geom.\ Funct.\ Anal.\ {\bf 20} (2010), 726--740.
%, {\tt http://www.math.harvard.edu/\~ctm/papers/home/text/papers/winning/winning.pdf}.


\bibitem{Mo} N.G.\ Moshchevitin, \textsl{A note on badly approximable affine forms and winning sets},
Moscow Math.\ J {\bf 11} (2011),  no.\ 1, 129--137.



\bibitem{PV}A.D.\ Pollington and S.L.\ Velani, 
\textsl{Metric Diophantine approximation and `absolutely friendly' measures},
Selecta Math.\ 
{\bf 11} (2005), 297--307. 


\bibitem{S1}W.M.\ Schmidt, \textsl{On badly approximable numbers and certain games}, 
Trans.\ A.M.S. {\bf 123} (1966), 27--50.


\bibitem{S2}   \bysame,  \textsl{Badly approximable systems of
linear forms}, J. Number Theory {\bf 1} (1969), 139--154.

\bibitem{S3}
  \bysame, 
\textsl{\da},
Lecture Notes in Mathematics, vol.\ 785, Springer-Verlag, Berlin, 1980.


\bibitem{SU} B.\ Stratmann and M.\ Urbanski, \textsl{Diophantine extremality of the Patterson measure},  Math.\ Proc.\ Cambridge Phil.\ Soc.\ {\bf 140} (2006), 297--304.



\bibitem{T1} J.\ Tseng, \textsl{Schmidt games and Markov partitions},  Nonlinearity  {\bf 22}  (2009),  no.\ 3, 525--543.


\bibitem{T2} \bysame,
\textsl{Badly approximable affine forms and Schmidt games}, 
J.\ Number Theory   {\bf 129} (2009), 3020--3025. 




\bibitem{U2} M.\ Urbanski,  \textsl{Diophantine approximation of self-conformal measures},  J.\ Number Theory {\bf 110} (2005), 219--235.



\bibitem{W} B.\ Weiss, \textsl{Almost no points on a Cantor set are very well approximable}, Proc.\ R.\ Soc.\ Lond.\ A  {\bf 457} (2001),  949--952.

\end{thebibliography}

\end{document}